\documentclass[12pt]{article}
\usepackage{amsthm,amsfonts,amssymb,amsmath}
\usepackage{stmaryrd}
\usepackage{cite,hyperref}
\usepackage{epsfig}
\usepackage{url}
\usepackage{footmisc}
\usepackage{xcolor,tikz}
\usetikzlibrary{positioning, arrows.meta,calc}
\usetikzlibrary{matrix}
\usepackage{multicol,graphicx}
\usepackage{fullpage}
\usepackage{bbm}
\usetikzlibrary{decorations}
\usepackage{svg}

\usepackage[shortlabels]{enumitem}

\numberwithin{equation}{section}

\newtheorem{thm}[equation]{Theorem}

\newtheorem{lem}[equation]{Lemma}
\newtheorem{prop}[equation]{Proposition}

\theoremstyle{definition}
\newtheorem{defn}[equation]{Definition}

\newtheorem{ex}[equation]{Example}

\newtheorem*{ack}{Acknowledgements}

\theoremstyle{remark}
\newtheorem{case}{Case}
\newtheorem{casee}{Case}

\newtheorem{rem}[equation]{Remark}

\title{Maximizing Alternating Paths via Entropy}
\author{Hao Chen\thanks{School of Mathematical Sciences, University of Science and Technology of China, Hefei, Anhui, 230026, China. E-mail: {\tt mathsch@mail.ustc.edu.cn}. This work was completed while the first author was visiting the University of Victoria. Research supported by National Key Research and Development Program of China 2023YFA1010201, National Natural Science Foundation of China grant 12125106 and China Scholarship Council No. 202406340066.} \and Felix Christian Clemen\thanks{Department of Mathematics and Statistics, University of Victoria, Victoria, B.C., Canada.}$\text{ }^{,}$\thanks{E-mail: {\tt fclemen@uvic.ca}. Research supported by a PIMS Postdoctoral Fellowship. } \and Jonathan A. Noel\footnotemark[2]$\text{ }^{,}$\thanks{E-mail: {\tt noelj@uvic.ca}. Research supported by NSERC Discovery Grant RGPIN-2021-02460.}}

\DeclareTextCompositeCommand{\v}{OT1}{l}{l\nobreak\hspace{-.1em}'}
\DeclareTextCompositeCommand{\v}{OT1}{t}{t\nobreak\hspace{-.1em}'\nobreak\hspace{-.15em}}

\DeclareMathOperator{\Hom}{Hom}

\DeclareMathOperator{\rng}{rng}

\begin{document}

\maketitle

\begin{abstract}
We prove that if $G$ is an $n$-vertex graph whose edges are coloured with red and blue, then the number of colour-alternating walks of length $2k+1$ with $k+1$ red edges and $k$ blue edges is at most $k^k(k+1)^{k+1}(2k+1)^{-2k-1}n^{2k+2}$. This solves a problem that was recently posed by Basit, Granet, Horsley, K\"undgen and Staden. Our proof involves an application of the entropy method.
\end{abstract}

\section{Introduction}

Let $H$ be a finite simple undirected graph whose edges are coloured with two colours, say red and blue; we refer to such a structure as an \emph{edge-coloured graph} and note that, in this paper, we always deal with edge-colourings using only red and blue unless stated otherwise. A natural extremal problem, raised recently by Basit et al.~\cite{Basit+25+} and known as the \emph{semi-inducibility problem}, asks for the maximum possible number of copies of an edge-colored graph $H$ in an edge-coloured graph $G$ on $n$ vertices. 

In addition to defining the problem and solving it for several small examples and infinite families of edge-coloured graphs, the authors of~\cite{Basit+25+} raised some intriguing problems regarding edge-coloured paths and cycles. For $k\geq1$, the \emph{alternating path} $P_k^A$ is the edge-coloured graph consisting of a path with $k+1$ vertices $v_0,v_1,\dots,v_k$ where pairs of consecutive vertices are adjacent, every vertex is incident to at most one edge of each colour and the first edge $v_0v_1$ is red. The semi-inducibility problem for alternating paths of even length was solved in~\cite[Theorem~1.3]{Basit+25+}; see Theorem~\ref{thm:even} below. Our main result (Theorem~\ref{thm:paths}) settles the case of odd alternating paths; in particular, we solve~\cite[Problem~9.2]{Basit+25+}. 

Before stating these theorems, we require a few definitions. For two edge-coloured graphs $H$ and $G$, a \emph{homomorphism} from $H$ to $G$ is a mapping $\varphi:V(H)\to V(G)$ with the property that, for every edge $uv$ of $H$, $f(u)f(v)$ is an edge of $G$ of the same colour as $uv$. We let $\Hom(H,G)$ denote the set of all such homomorphisms and $\hom(H,G)$ be $|\Hom(H,G)|$. The \emph{homomorphism density} of $H$ in $G$ is then defined to be
\[t(H,G):=\frac{\hom(H,G)}{v(G)^{v(H)}}\]
where $v(F)$ denotes the number of vertices of an edge-coloured graph $F$. In other words, $t(H,G)$ is the probability that a uniformly random map from $V(H)$ to $V(G)$ is a homomorphism. The result of~\cite{Basit+25+} on alternating paths of even length can be phrased as follows. 

\begin{thm}[Basit et al.~{\cite[Theorem~1.3]{Basit+25+}}]
\label{thm:even}
For every $k\geq1$ and edge-coloured graph $G$,
\[t(P_{2k}^A,G)\leq (1/2)^{2k}.\]
\end{thm}

Theorem~\ref{thm:even} can be seen to be asymptotically tight by letting $G$ be an $n$-vertex edge-coloured graphs in which every vertex is incident to approximately $n/2$ edges of each colour. We determine the semi-inducibility of odd alternating paths.

\begin{thm}
\label{thm:paths}
For every $k\geq1$ and edge-coloured graph $G$,
\[t(P_{2k+1}^A,G)\leq k^k(k+1)^{k+1}(2k+1)^{-2k-1}.\]
\end{thm}

Theorem~\ref{thm:paths} can be seen to be asympotically tight by taking $G$ to be an edge-coloured graph on $n$ vertices in which every vertex is incident to approximately $\frac{k+1}{2k+1}n$ red edges and approximately $\frac{k}{2k+1}n$ blue edges. While the semi-inducibility problem was originally phrased in~\cite{Basit+25+} in terms of ``copies'' of $H$ in $G$ as opposed to homomorphisms from $H$ to $G$, it is not hard to see that these two formulations are roughly equivalent (up to a constant factor and lower order asymptotic terms). We focus on homomorphisms here as we find them more convenient to deal with (e.g. it is useful that homomorphisms need not be injective, and that counting them does not require accounting for the number of ``symmetries'' of $H$).

In the next section, we provide an informal sketch of some of the key ideas of the proof, focusing mainly on the example of the 3-edge and 5-edge alternating paths. We also reduce Theorem~\ref{thm:paths} to proving a particular lemma about a special forest which we call $H_{2k+1}$. In Section~\ref{sec:entropy}, we review the basics of the entropy method and use it to prove a lemma that allows us, under certain conditions, to relate the homomorphism density of $P_{2k+1}^A$ to that of an edge-coloured forest $H$. Then, in Section~\ref{sec:construction}, we provide the rather technical construction of $H_{2k+1}$, show that it satisfies the condition of the lemma from the previous section, and bound its homomorphism density from above in terms of $P_{2k+1}^A$ to complete the proof. Some of the calculations related to this technical construction are included in an appendix. We then close the paper in Section~\ref{sec:conclusion} with some concluding remarks and ruminations about possible directions for future work. 

\begin{rem}
We have learned that Balogh, Lidick\'y, Mubayi and Pfender (personal communication) were working on problems related to semi-inducibility simultaneously to us and have obtained some related results. 
\end{rem}

\section{Proof Sketch and Warm-Up}
\label{sec:sketch}

A key idea in our proof of Theorem~\ref{thm:paths} is to construct a forest $H_{2k+1}$ such that the homomorphism density of $H_{2k+1}$ can be bounded from below and above in terms of the homomorphism density of $P_{2k+1}^A$, as described by the next lemma. For an edge-coloured graph $F$, we let $e_R(F)$ and $e_B(F)$ denote the number of red or blue edges of $F$, respectively, and $e(F)=e_R(F)+e_B(F)$. 

\begin{lem}
\label{lem:key}
For each $k\geq1$, there exists an edge-coloured forest $H_{2k+1}$ with more edges than $P_{2k+1}^A$ such that $k\cdot e_R(H_{2k+1})=(k+1)\cdot e_B(H_{2k+1})$ and every edge-coloured graph $G$ satisfies
\begin{equation}
\label{eq:PH}
t(P_{2k+1}^A,G)^{1/e(P_{2k+1}^A)}\leq t(H_{2k+1},G)^{1/e(H_{2k+1})}
\end{equation}
and 
\begin{equation}
\label{eq:HP}
t(H_{2k+1},G)\leq \left(k^{\frac{k}{2k+1}}(k+1)^{\frac{k+1}{2k+1}}(2k+1)^{-1}\right)^{e(H_{2k+1})-e(P_{2k+1}^A)}t(P_{2k+1}^A,G).
\end{equation}
\end{lem}

Note that, while many of the bounds in this paper are only tight when $G$ is an edge-coloured clique, it is interesting that \eqref{eq:PH} is tight for a much wider class of edge-coloured graphs. To see this, suppose that $p$ and $q$ are positive reals such that $p+q\leq 1$ and let $G$ be any edge-coloured graph on $n$ vertices in which every vertex is incident to about $pn$ red edges and about $qn$ blue edges. Since $P_{2k+1}^A$ and $H_{2k+1}$ are both forests, we get
\[t(P_{2k+1}^A,G)\approx p^{k+1}q^{k}\]
and
\[t(H_{2k+1},G)\approx p^{e_R(H_{2k+1})}q^{e_B(H_{2k+1})}=p^{\frac{k+1}{2k+1}e(H_{2k+1})}q^{\frac{k}{2k+1}e(H_{2k+1})}=t(P_{2k+1}^A,G)^{\frac{e(H_{2k+1})}{e(P_{2k+1}^A)}}.\]
and so \eqref{eq:PH} is tight for any such $G$. Let us demonstrate that Theorem~\ref{thm:paths} follows easily from Lemma~\ref{lem:key}, after which we will turn to sketching the proof of the lemma itself. 

\begin{proof}[Proof of Theorem~\ref{thm:paths}]
Let $G$ be an edge-coloured graph. By Lemma~\ref{lem:key}, we have
\begin{align*}t(P_{2k+1}^A,G)^{e(H_{2k+1})}&\leq t(H_{2k+1},G)^{e(P_{2k+1}^A)}\\&\leq \left(k^{\frac{k}{2k+1}}(k+1)^{\frac{k+1}{2k+1}}(2k+1)^{-1}\right)^{e(P_{2k+1}^A)\left(e(H_{2k+1})-e(P_{2k+1}^A)\right)}t(P_{2k+1}^A,G)^{e(P_{2k+1}^A)}.\end{align*}
The result now follows from a bit of algebra and the fact that $e(H_{2k+1})>e(P_{2k+1}^A)$. 
\end{proof}

A key aspect of $H_{2k+1}$ is that it is constructed by starting with $P_{2k+1}^A$ and adding pendant edges, short pendant paths, disjoint edges and disjoint vertices in such a way that $\hom(H_{2k+1},G)$ can be easily bounded above in terms of $\hom(P_{2k+1}^A,G)$ by analyzing local quantities, such as the degree sequence. Let us take a look at the construction in the most basic case $k=1$. Two suitable candidates for the edge-coloured forest $H_3$ are given in Figure~\ref{fig:H3} (as it turns out, there are many different forests which satisfy the conclusion of Lemma~\ref{lem:key}).

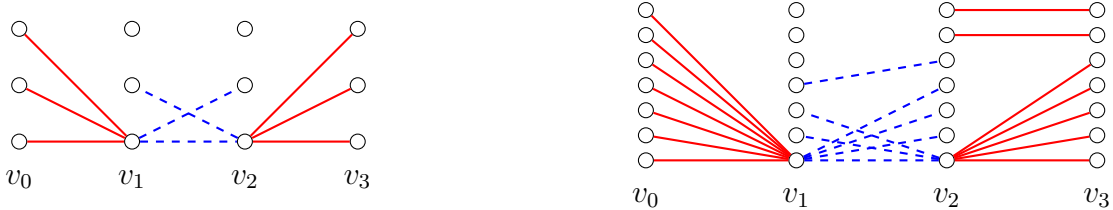
\begin{figure}[htbp]
\begin{center}
\begin{minipage}{0.45\textwidth}
\centering
\begin{tikzpicture}[scale=1.5, every node/.style={draw, circle, fill=white, inner sep=2pt}]
\node (v0) at (0, 0) {};
\node (v1) at (1, 0) {};
\node (v2) at (2, 0) {};
\node (v3) at (3, 0) {};
\node (x11) at (0,0.5){};
\node (x12) at (0,1){};
\node (x21) at (3,0.5){};
\node (x22) at (3,1){};
\node (x'1) at (2,0.5){};
\node (x'2) at (1,0.5){};
\node (w1) at (1,1){};
\node (w2) at (2,1){};

\node[below=5pt, draw=none, fill=none] at (v0) {$v_0$};
\node[below=5pt, draw=none, fill=none] at (v1) {$v_1$};
\node[below=5pt, draw=none, fill=none] at (v2) {$v_2$};
\node[below=5pt, draw=none, fill=none] at (v3) {$v_3$};

\draw[red, thick] (v0) -- (v1);
\draw[blue, dashed, thick] (v1) -- (v2);
\draw[red, thick] (v2) -- (v3);
\draw[red, thick] (v1) -- (x11);
\draw[red, thick] (v1) -- (x12);
\draw[red, thick] (v2) -- (x21);
\draw[red, thick] (v2) -- (x22);
\draw[blue, dashed, thick] (v1) -- (x'1);
\draw[blue, dashed, thick] (v2) -- (x'2);
\end{tikzpicture}
\end{minipage}
\hfill
\begin{minipage}{0.45\textwidth}
\centering
\begin{tikzpicture}[scale=2, every node/.style={draw, circle, fill=white, inner sep=2pt}]
\node (v0) at (0, 0) {};
\node (v1) at (1, 0) {};
\node (v2) at (2, 0) {};
\node (v3) at (3, 0) {};
\node (x11) at (0,0.16666666666){};
\node (x12) at (0,0.33333333333){};
\node (x13) at (0,0.5){};
\node (x14) at (0,0.66666666666){};
\node (x15) at (0,0.83333333333){};
\node (x16) at (0,1){};
\node (x21) at (3,0.16666666666){};
\node (x22) at (3,0.33333333333){};
\node (x23) at (3,0.5){};
\node (x24) at (3,0.66666666666){};
\node (x'11) at (2,0.16666666666){};
\node (x'12) at (2,0.33333333333){};
\node (x'13) at (2,0.5){};
\node (x'21) at (1,0.16666666666){};
\node (x'22) at (1,0.33333333333){};
\node (w11) at (1,1){};
\node (w12) at (1,0.83333333333){};
\node (w13) at (1,0.66666666666){};
\node (y1) at (1,0.5){};
\node (y2) at (2,0.66666666666){};
\node (y3) at (2,0.83333333333){};
\node (y4) at (3,0.83333333333){};
\node (y5) at (2,1){};
\node (y6) at (3,1){};

\node[below=5pt, draw=none, fill=none] at (v0) {$v_0$};
\node[below=5pt, draw=none, fill=none] at (v1) {$v_1$};
\node[below=5pt, draw=none, fill=none] at (v2) {$v_2$};
\node[below=5pt, draw=none, fill=none] at (v3) {$v_3$};

\draw[red, thick] (v0) -- (v1);
\draw[blue, dashed, thick] (v1) -- (v2);
\draw[red, thick] (v2) -- (v3);
\draw[red, thick] (v1) -- (x11);
\draw[red, thick] (v1) -- (x12);
\draw[red, thick] (v1) -- (x13);
\draw[red, thick] (v1) -- (x14);
\draw[red, thick] (v1) -- (x15);
\draw[red, thick] (v1) -- (x16);
\draw[red, thick] (v2) -- (x21);
\draw[red, thick] (v2) -- (x22);
\draw[red, thick] (v2) -- (x23);
\draw[red, thick] (v2) -- (x24);
\draw[blue, dashed, thick] (v1) -- (x'11);
\draw[blue, dashed, thick] (v1) -- (x'12);
\draw[blue, dashed, thick] (v1) -- (x'13);
\draw[blue, dashed, thick] (v2) -- (x'21);
\draw[blue, dashed, thick] (v2) -- (x'22);
\draw[blue, dashed, thick] (y1) -- (y2);
\draw[red, thick] (y3) -- (y4);
\draw[red, thick] (y5) -- (y6);
\end{tikzpicture}
\end{minipage}
\end{center}
\caption{Two possible choices of the edge-coloured forest $H_3$ in Lemma~\ref{lem:key}. Red edges are represented by solid lines and blue edges are represented by dashed lines.}
\label{fig:H3}
\end{figure}

For either of the candidates for $H_3$ in Figure~\ref{fig:H3}, the inequality \eqref{eq:HP} is easy to prove by analyzing vertex degrees. Indeed, given a vertex $v$ in an edge-coloured graph $G$, let $d_R(v)$ and $d_B(v)$ be the numbers of red and blue edges, respectively, incident to $v$. Consider the first coloured graph in Figure~\ref{fig:H3}. The vertices labelled $v_0,v_1,v_2$ and $v_3$ in the figure form a copy of $P_3^A$. So, for any $n$-vertex edge-coloured graph $G$ and any homomorphism from $H_3$ to $G$, the restriction of this homomorphism to $\{v_0,v_1,v_2,v_3\}$ is a homomorphism from $P_3^A$ to $G$. Given $f\in \Hom(P_3^A,G)$, let $\hom(H_3,G;f)$ be the number of homomorphisms from $H_3$ to $G$ whose restriction to $\{v_0,v_1,v_2,v_3\}$ is equal to $f$. Then, given that $\{v_0,v_1,v_2,v_3\}$ is mapped according to $f$, the number of choices for each isolated vertex of $H_3$ is clearly $n$, the number of choices of each red leaf adjacent to $v_1$, other than $v_0$, is $d_R(f(v_1))$, and so on. Thus,
\begin{align*}\hom(H_3,G) &= \sum_{f\in \Hom(P_3^A,G)}\hom(H_3,G;f)\\&= \sum_{f\in \Hom(P_3^A,G)}n^2d_R(f(v_1))^2d_B(f(v_1))d_R(f(v_2))^2d_B(f(v_2)).\end{align*}
Now, clearly, every vertex $u$ of $G$ satisfies $d_R(u)+d_B(u)\leq n-1$, which implies\footnote{Here, and throughout the paper, we frequently use the following inequality that is easy to prove with basic Calculus: If $a,b,m>0$ and $x,y\geq0$ are such that $x+y\leq m$, then $x^ay^b\leq m^{a+b}\left(\frac{a}{a+b}\right)^a\left(\frac{b}{a+b}\right)^b$.\label{calculus}} that $d_R(u)^2d_B(u)\leq (n-1)^3(1/3)(2/3)^2$. So,
\[\hom(H_3,G)\leq n^8(1/3)^2(2/3)^4\hom(P_3^A,G)\]
which, after dividing both sides by $n^{v(H_3)}=n^{v(P_3^A)+8}$, tells us that $H_3$ satisfies the inequality \eqref{eq:HP} in Lemma~\ref{lem:key}. The argument for the other candidate for $H_3$ in Figure~\ref{fig:H3} is fairly similar. For larger values of of $k$, the proof of \eqref{eq:HP} follows from similar principles, but the constructions and local inequalities that we use are slightly more complicated. To get a taste for the level of complexity in these constructions, see the candidate for $H_5$ in Figure~\ref{fig:H5}.

\begin{figure}[htbp]
{\begin{center}
\begin{tikzpicture}[scale=1.65, every node/.style={draw, circle, fill=white, inner sep=2pt}]

\node (v0) at (0, 0) {};
\node (v1) at (1, 0) {};
\node (v2) at (2, 0) {};
\node (v3) at (3, 0) {};
\node (v4) at (4, 0) {};
\node (v5) at (5, 0) {};
\node (x1) at (0,0.5){};
\node (x2) at (3,1){};
\node (x3) at (2,1){};
\node (x4) at (5,0.5){};
\node (x1') at (2,0.5){};
\node (x2') at (1,0.5){};
\node (x3') at (4,0.5){};
\node (x4') at (3,0.5){};
\node (z21) at (3,1.5){};
\node (z31) at (2,1.5){};
\node (z22) at (4,1.5){};
\node (z32) at (1,1.5){};
\node (z23) at (1,1){};
\node (z33) at (4,1){};
\node (z24) at (2,2){};
\node (z34) at (3,2){};
\node (y0) at (0,2){};
\node (y1) at (1,2){};
\node (y2) at (2,2.5){};
\node (y3) at (3,2.5){};
\node (y4) at (4,2){};
\node (y5) at (5,2){};

\node[below=5pt, draw=none, fill=none] at (v0) {$v_0$};
\node[below=5pt, draw=none, fill=none] at (v1) {$v_1$};
\node[below=5pt, draw=none, fill=none] at (v2) {$v_2$};
\node[below=5pt, draw=none, fill=none] at (v3) {$v_3$};
\node[below=5pt, draw=none, fill=none] at (v4) {$v_4$};
\node[below=5pt, draw=none, fill=none] at (v5) {$v_5$};
\node[above=1pt, draw=none, fill=none] at (0.5,2) {36};
\node[above=1pt, draw=none, fill=none] at (4.5,2) {36};
\node[above=1pt, draw=none, fill=none] at (2.5,2.5) {3};
\node[above=1pt, draw=none, fill=none] at (x1) {39};
\node[above=1pt, draw=none, fill=none] at (x4) {39};
\node[left=1pt, draw=none, fill=none] at (x1') {26};
\node[right=1pt, draw=none, fill=none] at (x4') {26};
\node[right=1pt, draw=none, fill=none] at (x2) {21};
\node[left=1pt, draw=none, fill=none] at (x3) {21};
\node[left=1pt, draw=none, fill=none] at (x2') {14};
\node[right=1pt, draw=none, fill=none] at (x3') {14};
\node[left=1pt, draw=none, fill=none] at (z32) {15};
\node[right=1pt, draw=none, fill=none] at (z22) {15};
\node[above=1pt, draw=none, fill=none] at (z34) {10};
\node[above=1pt, draw=none, fill=none] at (z24) {10};

\draw[red, thick] (v0) -- (v1); 
\draw[blue, dashed, thick] (v1) -- (v2); 
\draw[red, thick] (v2) -- (v3); 
\draw[blue, dashed, thick] (v3) -- (v4); 
\draw[red, thick] (v4) -- (v5); 

\draw[red, thick] (v1) -- (x1); 
\draw[red, thick] (v2) -- (x2); 
\draw[red, thick] (v3) -- (x3); 
\draw[red, thick] (v4) -- (x4); 
\draw[blue, dashed, thick] (v1) -- (x1'); 
\draw[blue, dashed, thick] (v2) -- (x2'); 
\draw[blue, dashed, thick] (v3) -- (x3'); 
\draw[blue, dashed, thick] (v4) -- (x4'); 
\draw[red, thick] (v3) -- (z31); 
\draw[red, thick] (v2) -- (z21); 
\draw[blue, dashed, thick] (z31) -- (z32); 
\draw[blue, dashed, thick] (z21) -- (z22); 
\draw[blue, dashed, thick] (v3) -- (z33); 
\draw[blue, dashed, thick] (v2) -- (z23); 
\draw[blue, dashed, thick] (z33) -- (z34); 
\draw[blue, dashed, thick] (z23) -- (z24); 
\draw[red, thick] (y0) -- (y1); 
\draw[red, thick] (y2) -- (y3); 
\draw[red, thick] (y4) -- (y5); 
\end{tikzpicture}
\end{center}}
    \caption{A representation of a possible choice of the edge-coloured forest $H_5$ in Lemma~\ref{lem:key}. Each isolated edge is labelled by a number which represents the number of copies of that edge added in the construction. Each leaf that is not in an isolated edge, apart from $v_0$ and $v_5$, is labelled with a number. This number represents the number of copies of the unique path from that leaf to $\{v_0,\dots,v_5\}$ that are added in the construction of $H_5$ from $P_5^A$. For example, $H_5$ contains exactly 14 leaves which are adjacent to $v_3$ via a blue edge, and exactly 15 paths of length two starting with $v_2$ in which the first edge is red and the second edge is blue and the final vertex is not $v_4$.}
    \label{fig:H5}
\end{figure}
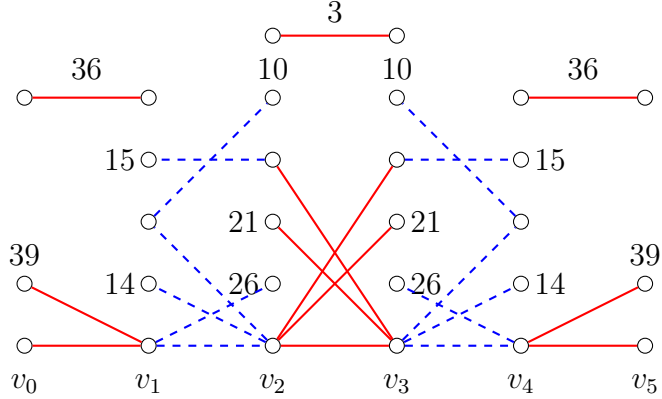

What remains is to address the inequality \eqref{eq:PH}. The proof of this will use the concept of entropy from information theory. The application of entropy in proving inequalities for graph homomorphism counts was pioneered in a paper of Kopparty and Rossman~\cite{KoppartyRossman11} and has become an indispensable tool in the area; see, e.g.,~\cite{Lee21,BehagueCrudeleNoelSimbaqueba+25+,BehagueMorrisonNoel24,ConlonKimLeeLee18,Szegedy15,BlekhermanRaymond23,BlekhermanRaymond22,GrzesikLeeLidickyVolec22,ChaoYu24}. Our approach is particularly closely related to the ideas used in~\cite{BehagueCrudeleNoelSimbaqueba+25+} to prove inequalities analogous to \eqref{eq:PH} for pairs of trees without edge colours. 

To explain the proof of \eqref{eq:PH} in detail, even in the case $k=1$, requires a formal introduction to entropy and its basic properties, which is provided in Section~\ref{sec:entropy}. To keep this sketch light, we will only give a vague idea of the proof which highlights the key properties of $H_{2k+1}$ that we require. Some of the technical jargon related to entropy and probability theory is written in quotation marks to indicate that these terms have not been properly defined yet, but they will be in Section~\ref{sec:entropy}.

Consider the second edge-coloured graph $H_3$ in Figure~\ref{fig:H3}. Let $G$ be an edge-coloured graph with $\hom(P_3^A,G)\geq1$, let $f$ be a uniformly random homomorphism from $P_3^A$ to $G$ and consider the random variable $(X_0,X_1,X_2,X_3)$ where $X_i:=f(v_i)$ for $0\leq i\leq 3$. Of course, the random variables $f$ and $(X_0,X_1,X_2,X_3)$ contain precisely the same information (i.e. each is uniquely determined by the other). Now, in order to prove that $\hom(H_3,G)\geq\hom(P_3^A,G)^7$ (which is equivalent to \eqref{eq:PH} in this case), it suffices to construct a distribution on $\Hom(H_3,G)$ with ``entropy'' at least seven times that of $(X_0,X_1,X_2,X_3)$. In constructing this distribution, we start by mapping vertices $(v_0,v_1,v_2,v_3)$ according to the random variable $(X_0,X_1,X_2,X_3)$; this will preserve all edges and edge colours between these vertices since $f$ was a homomorphism from $P_3^A$ to $G$. Next, for each of the six leaves of $H_3$ adjacent to $v_1$, other than $v_0$, we map them according to the distribution of $X_0$, ``conditioned on'' the choice of $X_1$, in such a way that this random variable is ``conditionally independent'' of all other previous choices given $X_1$. The mapping of the other leaves adjacent to $v_1$ and $v_2$, other than $v_0$ and $v_3$, is done analogously. The two vertices of the isolated blue edge are mapped according to $(X_1,X_2)$ and the isolated red edges are both mapped according to $(X_2,X_3)$ in a way that is independent of all other choices made so far. Finally, the three isolated vertices are each mapped according to the distribution of $X_1$, independently of previous choices. 

As it turns out, the entropy of this distribution can be shown to be exactly seven times that of $(X_0,X_1,X_2,X_3)$ which allows us to use standard facts about entropy to conclude that \eqref{eq:PH} holds. To prove this correspondence between the entropy of this distribution and that of $(X_0,X_1,X_2,X_3)$, what we need is the existence of a homomorphism from $H_3$ to $P_3^A$ which ``covers'' every vertex and edge of $P_3^A$ exactly seven times. The drawing of $H_3$ in Figure~\ref{fig:H3} was chosen to make this mapping easy to see; if we simply map each vertex of $H_3$ to the vertex $v_0,v_1,v_2$ or $v_3$ drawn in the same column as it, then we get a homomorphism from $H_3$ to $P_3^A$ in which the preimage of each vertex and edge of $P_3^A$ consists of exactly seven vertices or edges of $H_3$. The other candidate for $H_3$ in Figure~\ref{fig:H3} has a similar property, except that each vertex and edge is covered three times, and so this yields a distribution on $\Hom(H_3,G)$ with entropy exactly three times the entropy of $(X_0,X_1,X_2,X_3)$. A similar idea can be used for the example of $H_5$ in Figure~\ref{fig:H5} because $H_5$ admits a homomorphism to $P_5^A$ covering every vertex and edge exactly 76 times. This concludes our brief sketch of the proof. 

To close this section, we remark that a similar, but easier, construction can be used to get an alternative proof of the inequality $t(P_{2k}^A,G)\leq (1/2)^{2k}$ in Theorem~\ref{thm:even} proved in~\cite{Basit+25+}. In this case, one can let $H_{2k}$ be obtained from $P_{2k}$ by adding, for each $1\leq j\leq 2k-1$, one vertex adjacent to $v_j$ via a red edge and one adjacent to $v_j$ via a blue edge and then adding an isolated edge of each colour. See Figure~\ref{fig:H2k} for an illustration. The forest $H_{2k}$ satisfies inequalities analogous to \eqref{eq:PH} and \eqref{eq:HP} for the path $P_{2k}^A$ of even length.

\begin{figure}[htbp]
{\begin{center}
\begin{tikzpicture}[scale=1.65, every node/.style={draw, circle, fill=white, inner sep=2pt}]

\node (v0) at (0, 0) {};
\node (v1) at (1, 0) {};
\node (v2) at (2, 0) {};
\node (v3) at (3, 0) {};
\node (v4) at (5, 0) {};
\node (v5) at (6, 0) {};
\node (v6) at (7, 0) {};
\node (v7) at (8, 0) {};

\node (u1) at (0.8, 1) {};
\node (u2) at (1.8, 1) {};
\node (u3) at (2.8, 1) {};
\node (u4) at (4.8, 1) {};
\node (u5) at (5.8, 1) {};
\node (u6) at (6.8, 1) {};

\node (w1) at (1.2, 1) {};
\node (w2) at (2.2, 1) {};
\node (w3) at (3.2, 1) {};
\node (w4) at (5.2, 1) {};
\node (w5) at (6.2, 1) {};
\node (w6) at (7.2, 1) {};

\node (z1) at (1, -1) {};
\node (z2) at (2, -1) {};
\node (z3) at (6, -1) {};
\node (z4) at (7, -1) {};

\node[below=5pt, draw=none, fill=none] at (v0) {$v_0$};
\node[below=5pt, draw=none, fill=none] at (v1) {$v_1$};
\node[below=5pt, draw=none, fill=none] at (v2) {$v_2$};
\node[below=5pt, draw=none, fill=none] at (v3) {$v_3$};
\node[below=1.2pt, draw=none, fill=none] at (v4) {$v_{2k-3}$};
\node[below=1.2pt, draw=none, fill=none] at (v5) {$v_{2k-2}$};
\node[below=1.2pt, draw=none, fill=none] at (v6) {$v_{2k-1}$};
\node[below=5pt, draw=none, fill=none] at (v7) {$v_{2k}$};

\draw[red, thick] (v0) -- (v1); 
\draw[blue, dashed, thick] (v1) -- (v2); 
\draw[red, thick] (v2) -- (v3);
\draw[blue, dashed, thick] (v4) -- (v5); 
\draw[red, thick] (v5) -- (v6); 
\draw[blue, dashed, thick] (v6) -- (v7); 

\draw[red, thick] (v1) -- (u1);
\draw[red, thick] (v2) -- (u2);
\draw[red, thick] (v3) -- (u3);
\draw[red, thick] (v4) -- (u4);
\draw[red, thick] (v5) -- (u5);
\draw[red, thick] (v6) -- (u6);

\draw[blue, dashed, thick] (v1) -- (w1);
\draw[blue, dashed, thick] (v2) -- (w2);
\draw[blue, dashed, thick] (v3) -- (w3);
\draw[blue, dashed, thick] (v4) -- (w4);
\draw[blue, dashed, thick] (v5) -- (w5);
\draw[blue, dashed, thick] (v6) -- (w6);

\draw[blue, dashed, thick] (z1) -- (z2);
\draw[red, thick] (z3) -- (z4);

\fill (3.85,0) circle (0.5pt);
\fill (4.0,0) circle (0.5pt);
\fill (4.15,0) circle (0.5pt);
\end{tikzpicture}
\end{center}}
    \caption{The edge-colored graph $H_{2k}$.}
    \label{fig:H2k}
\end{figure}
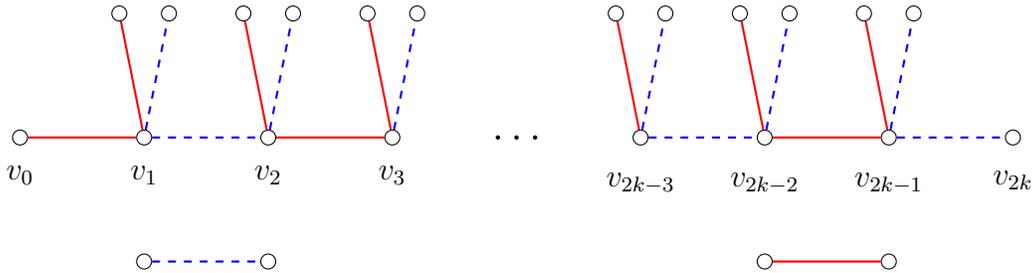

\section{Homomorphism Counting Inequalities Via Entropy}
\label{sec:entropy}

Our focus in this section is on proving a lemma which will allow us to establish the inequality \eqref{eq:PH} in Lemma~\ref{lem:key}. To keep the calculations simple and focused, we will only deal with the case of odd alternating paths; however, by incorporating more of the ideas from~\cite{BehagueCrudeleNoelSimbaqueba+25+} (which are built upon ideas from~\cite{KoppartyRossman11}), one can extend this approach to edge-coloured trees and forests more generally. Some of the ideas are also useful beyond the setting of forests (see, e.g.,~\cite{BehagueMorrisonNoel23+,KoppartyRossman11,ConlonKimLeeLee18}). Further discussion on these possible directions can be found in the conclusion.

The key definition that we need is the notion of the entropy of a discrete random variable. All of the lemmas about entropy stated without proof in this section are standard; many of them can be found in, e.g.,~\cite[Section~15.7]{AlonSpencer}.

\begin{defn}
For a discrete random variable $X$, define the \emph{range} of $X$ to be $\rng(X):=\{x: \mathbb{P}(X=x)>0\}$. 
\end{defn}

\begin{defn}
\label{defn:entropy}
The \emph{entropy} of a discrete random variable $X$ is 
\[\mathbb{H}(X):=-\sum_{x\in \rng(X)}\mathbb{P}(X=x)\log_2(\mathbb{P}(X=x)).\]
\end{defn}

An important property of entropy is that, among all random variables with a given range, the one with largest entropy is the uniform distribution. This can be proved via a simple application of Jensen's Inequality. 

\begin{lem}[Maximality of the Uniform Distribution]
\label{lem:unif}
If $X$ is a random variable with finite range, then
\[\mathbb{H}(X)\leq \log_2(|\rng(X)|)\]
with equality if and only if $\mathbb{P}(X=x)=\frac{1}{|\rng(X)|}$ for all $x\in \rng(X)$. 
\end{lem}

We also require the notion of conditional entropy. 

\begin{defn}
For two discrete random variables $X$ and $Y$ and $y\in \rng(Y)$, define the \emph{range of $X$ given that $Y=y$} to be $\rng(X\mid Y=y):=\{x: \mathbb{P}(X=x\mid Y=y)>0\}$.
\end{defn}

\begin{defn}
Given discrete random variables $X$ and $Y$ and $y\in\rng(Y)$, the \emph{conditional entropy of $X$ given that $Y=y$} is 
\[\mathbb{H}(X\mid Y=y):=-\sum_{x\in \rng(X\mid Y=y)}\mathbb{P}(X=x\mid Y=y)\log_2(\mathbb{P}(X=x\mid Y=y)).\]
\end{defn}

\begin{defn}
For discrete random variables $X$ and $Y$, the \emph{conditional entropy} of $X$ given $Y$ is
\[\mathbb{H}(X\mid Y)=\sum_{y\in \rng(Y)}\mathbb{P}(Y=y)\mathbb{H}(X\mid Y=y).\]
\end{defn}

The next lemma, known as the ``Chain Rule'' for entropy, is useful for breaking down the entropy of a tuple of random variables into a sum of conditional entropy expressions. The proof just relies on Bayes' formula for conditional probability and basic properties of logarithms. 

\begin{lem}[Chain Rule]
\label{lem:chain}
For any discrete random variables $X_1,\dots,X_n$, 
\[\mathbb{H}(X_1,\dots,X_n)=\mathbb{H}(X_1)+\sum_{i=2}^n\mathbb{H}(X_i\mid X_1,\dots,X_{i-1}).\]
\end{lem}

We will require an analogue of Lemma~\ref{lem:unif} for conditional entropy. This requires the notion of conditional independence of random variables, which plays the same role in this lemma as the uniform distribution plays in Lemma~\ref{lem:unif}.

\begin{defn}
\label{def:condIndep}
Let $X,Y$ and $Z$ be three discrete random variables. We say that $X$ and $Z$ are \emph{conditionally independent given $Y$} if
\[\mathbb{P}(X=x\mid Y=y, Z=z)=\mathbb{P}(X=x\mid Y=y)\]
and
\[\mathbb{P}(Z=z\mid Y=y, X=x)=\mathbb{P}(Z=z\mid Y=y)\]
for all $y\in\rng(Y)$, $x\in \rng(X\mid Y=y)$ and $z\in\rng(Z\mid Y=y)$. 
\end{defn}

The next lemma can be proved using Jensen's Inequality in a way that is similar to the proof of Lemma~\ref{lem:unif}. 

\begin{lem}[Deconditioning Lemma]
\label{lem:condIndep}
For any discrete random variables $X,Y$ and $Z$, we have
\[\mathbb{H}(X\mid Y,Z)\leq \mathbb{H}(X\mid Y)\]
where equality holds if and only if $X$ and $Z$ are conditionally independent given $Y$. 
\end{lem}

Using the Chain Rule and conditional independence, we can now obtain a particularly nice expression for the entropy of a uniformly random homomorphism from $P_{2k+1}^A$ to an edge-coloured graph $G$. A similar formula holds for any edge-coloured forest. 

\begin{prop}
\label{prop:pathEntropy}
For $k\geq1$, let $G$ be an edge-coloured graph with $\hom(P_{2k+1}^A,G)\geq1$. Let $f$ be a uniformly random element of $\Hom(P_{2k+1}^A,G)$ and, for $0\leq i\leq 2k+1$, let $X_i:=f(v_i)$. Then the entropy of $f$ is equal to
\[\sum_{i=0}^{2k}\mathbb{H}(X_i,X_{i+1}) - \sum_{i=1}^{2k}\mathbb{H}(X_i).\]
\end{prop}

\begin{proof}
The random variables $f$ and $(X_0,X_1,\dots,X_{2k+1})$ encode the same information (i.e. each of them is uniquely determined by the other) and so they have the same entropy, simply by Definition~\ref{defn:entropy}. So, by the Chain Rule (Lemma~\ref{lem:chain}),
\begin{equation}\label{eq:CRpath}\mathbb{H}(f) = \mathbb{H}(X_0) + \sum_{i=1}^{2k+1}\mathbb{H}(X_i\mid X_0,X_1,\dots,X_{i-1}).\end{equation}
Now, we analyze the terms $\mathbb{H}(X_i\mid X_0,X_1,\dots,X_{i-1})$ for $i\geq2$. Obviously, the choice of $X_i$ must depend on the choice of $X_{i-1}$, as these two vertices must be adjacent in $G$ via an edge of the correct colour. However, using the fact that $f$ is uniform, we observe that the choices of $X_i$ and $(X_0,\dots,X_{i-2})$ are conditionally independent given $X_{i-1}$. That is, if we have prior knowledge of the outcome of $X_{i-1}$, the probability of any given outcome of $X_i$ is agnostic to the outcomes of $X_0,\dots,X_{i-2}$, and vice versa. Therefore, \eqref{eq:CRpath} can be rewritten as 
\[\mathbb{H}(f) = \mathbb{H}(X_0) + \sum_{i=1}^{2k+1}\mathbb{H}(X_i\mid X_{i-1}).\]
Applying the Chain Rule to $\mathbb{H}(X_{i-1},X_i)$ for $1\leq i\leq 2k+1$ and doing a bit of rearranging yields $\mathbb{H}(X_i\mid X_{i-1}) = \mathbb{H}(X_i,X_{i-1}) -\mathbb{H}(X_{i-1})$. The result now follows by substituting this in, cancelling the $\mathbb{H}(X_0)$ terms and shifting the index of the summation.
\end{proof}

So far, we have developed plenty of tools for analyzing the entropy of a given random variable, but, to prove inequalities as in \eqref{eq:PH}, we will also need to be able to build one tuple of random variables from another in such a way that the entropy of the new random variables can be compared to the entropy of those that they are built from. The following lemma is a standard tool for achieving this; see, e.g.,~\cite[Lemma~2.5]{Lee21}. We include a proof for completeness.

\begin{lem}[Distribution Gluing Lemma]
\label{lem:glue}
Let $(A,B_1,B_2,C)$ be a tuple of discrete random variables such that $B_1$ and $B_2$ are identically distributed. Then there exists a random variable $C'$ such that $(B_1,C')$ and $(B_2,C)$ are identically distributed and $A$ and $C'$ are conditionally independent given $B_1$. 
\end{lem}

\begin{proof}
First, sample the pair $(A,B_1)$ according to its joint distribution. Next, given the outcome of $(A,B_1)$, we sample $C'$ according to the distribution of $C$ conditioned on $B_2=B_1$. That is, if, after sampling $(A,B_1)$, we have $B_1=b$, then we choose $C'$ randomly where $C'$ is equal to $c$ with probability $\mathbb{P}(C=c\mid B_2=b)$ for all $c\in \rng(C\mid B_2=b)$. Note that the expression $\mathbb{P}(C=c\mid B_2=b)$ is well-defined with probability one because $B_1$ and $B_2$ are identically distributed discrete random variables, and so the event $B_1=b$ has no probability of occurring unless $\mathbb{P}(B_1=b)=\mathbb{P}(B_2=b)>0$. 

We verify that $C'$ has the desired properties. For any $b,c$ where $b\in \rng(B_1)$, we have 
\begin{align*}
\mathbb{P}(B_1=b,C'=c)&= \mathbb{P}(B_1=b)\mathbb{P}(C'=c\mid B_1=b)\\
&=\mathbb{P}(B_2=b)\mathbb{P}(C=c\mid B_2=b)\\
&=\mathbb{P}(B_2=b,C=c)
\end{align*}
where the penultimate equality used the definition of $C'$ and the fact that $B_1$ and $B_2$ are identically distributed. Thus, $(B_1,C')$ and $(B_2,C)$ are identically distributed. For any $b\in\rng(B_1)$, $a\in \rng(A\mid B_1=b)$ and $c\in \rng(C'\mid B_1=b)$, we have
\[\mathbb{P}(A=a\mid B_1=b,C'=c) = \mathbb{P}(A=a\mid B_1=b)\]
simply because $A$ was sampled before $C'$, with no prior knowledge of the outcome of $C'$, and therefore conditioning on the outcome of $C'$ cannot affect $A$ if we have also conditioned on the outcome of $B_1$. Also, by construction of $C'$, we have
\[\mathbb{P}(C'=c\mid B_1=b,A=a) = \mathbb{P}(C=c\mid B_2=b)=\mathbb{P}(C'=c\mid B_1=b).\]
Thus, $A$ and $C'$ are conditionally independent given $B_1$.
\end{proof}

Before proving the general lemma that will be used to prove \eqref{eq:PH} for all $k$, let us first illustrate how it works by providing a more rigorous definition of the random variable that was constructed for the warm-up example in Section~\ref{sec:sketch} and analyzing its entropy. 

\begin{ex}
\label{ex:entropyEx}
Recall that $H_3$ is the forest in the second picture in Figure~\ref{fig:H3}, $f$ is a uniformly random homomorphism from $P_3^A$ to an edge-coloured graph $G$ and that $X_i:=f(v_i)$ for $0\leq i\leq 3$. Since $f$ is uniform and $(X_0,X_1,X_2,X_3)$ encodes the same information as $f$, we have, by Lemma~\ref{lem:unif}, that
\begin{equation}\label{eq:unifEx}\mathbb{H}(X_0,X_1,X_2,X_3) = \mathbb{H}(f)=\log_2(\hom(P_3^A,G)).\end{equation}
We construct a tuple of random variables $(Y_v: v\in V(H_3))$ indexed by the vertices of $H_3$ with the properties that
\begin{enumerate}[(a)]
    \item\label{eq:homSupported} the map $g:V(H_3)\to V(G)$ defined by $g(v)=Y_v$ is a homomorphism from $H_3$ to $G$ with probability one and
    \item\label{eq:7} $\mathbb{H}(Y_v: v\in V(H_3)) = 7\cdot \mathbb{H}(X_0,X_1,X_2,X_3)$.
\end{enumerate}
The construction is built up iteratively as follows. First, we let $(Y_{v_0},Y_{v_1},Y_{v_2},Y_{v_3})$ be chosen randomly with the same distribution as $(X_0,X_1,X_2,X_3)$. 

Next, consider a leaf $v$ of $H_3$ that is adjacent to $v_1$ via a red edge. To define $Y_v$, we apply Lemma~\ref{lem:glue} with $A=(Y_{v_0},Y_{v_2},Y_{v_3})$, $B_1=Y_{v_1}$, $B_2=X_1$ and $C=X_0$ and take $Y_v$ to be the random variable $C'$ resulting from the lemma. Note that this is a valid application of the lemma because $Y_{v_1}$ has the same distribution as $X_1$. From the lemma, we will have that $(Y_v,Y_{v_1})$ has the same distribution as $(X_0,X_1)$ and that $Y_v$ and $(Y_{v_0},Y_{v_2},Y_{v_3})$ are conditionally independent given $Y_{v_1}$. Therefore, by Lemmas~\ref{lem:chain} and~\ref{lem:condIndep},
\[\mathbb{H}(Y_{v_0},Y_{v_1},Y_{v_2},Y_{v_3},Y_v) = \mathbb{H}(X_0,X_1,X_2,X_3) + \mathbb{H}(X_0\mid X_1)\]
which, by another application of the Chain Rule on the second term, is equal to 
\[\mathbb{H}(Y_{v_0},Y_{v_1},Y_{v_2},Y_{v_3},Y_v) = \mathbb{H}(X_0,X_1,X_2,X_3) + \mathbb{H}(X_0, X_1) - \mathbb{H}(X_1).\]
In an analogous way,  by repeatedly applying Lemma~\ref{lem:glue}, we can continue to build up $Y_v$ for all of the rest of the leaves $v$ of $H_3$ that are adjacent to $v_1$ in such a way that every such variable $Y_v$ is conditionally independent of all other variables, apart from $Y_{v_1}$, given $Y_{v_1}$. The same thing can then be done for all of the leaves that are adjacent to $v_2$. After all of these steps, the entropy of the tuple of all variables chosen so far is
\begin{align*}\mathbb{H}(X_0,X_1,X_2,X_3)& + 6\cdot\mathbb{H}(X_0, X_1) + 3\cdot\mathbb{H}(X_1,X_2)- 9\cdot\mathbb{H}(X_1) \\&+ 4\cdot\mathbb{H}(X_2,X_3)+2\cdot\mathbb{H}(X_1,X_2) - 6\cdot\mathbb{H}(X_2).
\end{align*}
After this, if $u$ and $v$ are the two vertices of the isolated blue edge of $H_3$, then we take $(Y_u,Y_v)$ to be distributed in the same way as $(X_1,X_2)$, independent of all previous choices. Similarly, the vertices of each of the two isolated red edges correspond to independent copies of $(X_2,X_3)$. Finally, for each of the three isolated vertices $v$, we let $Y_v$ be a copy of $X_1$ chosen independently of everything that was done so far. This completes the definition of $(Y_v: v\in V(H_3))$. It is clear that all of these choices maintain property \ref{eq:homSupported}. In the end, the entropy of this random variable is 
\begin{align*}\mathbb{H}(X_0,X_1,X_2,X_3)& + 6\cdot\mathbb{H}(X_0, X_1) + 3\cdot\mathbb{H}(X_1,X_2)- 9\cdot\mathbb{H}(X_1) \\&+ 4\cdot\mathbb{H}(X_2,X_3)+2\cdot\mathbb{H}(X_1,X_2) - 6\cdot\mathbb{H}(X_2)\\&+\mathbb{H}(X_1,X_2)+2\cdot\mathbb{H}(X_2,X_3)+3\cdot\mathbb{H}(X_1)
\end{align*}
which, after collecting like terms, is equal to $\mathbb{H}(X_0,X_1,X_2,X_3)$ plus
\[6\cdot \mathbb{H}(X_0,X_1)+6\cdot\mathbb{H}(X_1,X_2)+6\cdot\mathbb{H}(X_2,X_3)+6\cdot\mathbb{H}(X_3,X_4) - 6\cdot\mathbb{H}(X_1) - 6\cdot\mathbb{H}(X_2).\]
By Proposition~\ref{prop:pathEntropy}, the above expression is equal to $6\cdot\mathbb{H}(X_0,X_1,X_2,X_3)$. Thus, \ref{eq:7} is satisfied. Now, by Lemma~\ref{lem:unif}, \eqref{eq:unifEx} and properties \ref{eq:homSupported} and \ref{eq:7},
\[\log_2(\hom(H_3,G)) \geq \mathbb{H}(Y_v:v\in V(H_3)) = 7\cdot\mathbb{H}(X_0,X_1,X_2,X_3) = 7\cdot\log_2(\hom(P_3^A,G))\]
and so $\hom(H_3,G)\geq \hom(P_3^A,G)^7$. Now, dividing both sides by $n^{v(H_3)} = n^{7\cdot v(P_3^A)}$ yields the inequality in \eqref{eq:PH}.
\end{ex}

We conclude this section with a lemma which generalizes the previous example. In order to state it, we need a couple of definitions. Given edge-coloured graphs $H$ and $F$, a vertex $v\in V(F)$, a homomorphism $\varphi$ from $H$ to $F$ and an integer $m$, we say that $\varphi$ \emph{covers $v$ exactly $m$ times} if $|\varphi^{-1}(v)|=m$. Analogously, for an edge $e=uv\in E(F)$, we say that $\varphi$ \emph{covers $e$ exactly $m$ times} if the set $\varphi^{-1}(e):=\{ab\in E(H): \{\varphi(a),\varphi(b)\}=\{u,v\}\}$ has cardinality $m$. 

\begin{lem}
\label{lem:entropy}
If $H$ is a non-empty edge-coloured forest such that there exists a homomorphism $\varphi$ from $H$ to $P_{2k+1}^A$ which covers every edge and vertex of $P_{2k+1}^A$ exactly $e(H)/e(P_{2k+1}^A)$ times, then
\[t(P_{2k+1}^A,G)^{1/e(P_{2k+1}^A)}\leq t(H,G)^{1/e(H)}\]
for every edge-coloured graph $G$.
\end{lem}

\begin{proof}
Let $G$ be an arbitrary edge-coloured graph. In the case that $\hom(P_{2k+1}^A,G)=0$ the inequality in the lemma holds trivially. So, from here forward, we assume that $\hom(P_{2k+1}^A,G)\geq1$. Thus, we can let $f$ be a uniformly random homomorphism from $P_{2k+1}^A$ to $G$ and, for $0\leq i\leq 2k+1$, define $X_i:=f(v_i)$. As in Example~\ref{ex:entropyEx}, our goal is to construct a random variable $(Y_v: v\in V(H))$ such that
\begin{enumerate}[(a)]
    \item\label{eq:homSupportedGeneral} the map $g:V(H)\to V(G)$ defined by $g(v)=Y_v$ is a homomorphism from $H$ to $G$ with probability one and
    \item\label{eq:same} $\mathbb{H}(Y_v: v\in V(H)) = (e(H)/e(P_{2k+1}^A))\cdot \mathbb{H}(X_0,\dots,X_{2k+1})$.
\end{enumerate}

For each component $C$ of $H$, let $r_C$ be an arbitrary vertex of $V(C)$, called the \emph{root} of $C$. For each non-root vertex $v$ of $H$, let $p(v)$ be the unique neighbour of $v$ which is closer to a root than $v$ is; we call $p(v)$ the \emph{parent} of $v$ and say that $v$ is the \emph{child} of $p(v)$.

Now, we build a random variable $(Y_v: v\in V(H))$ by repeatedly applying Lemma~\ref{lem:glue} as follows. For each component $C$, one by one, start by selecting $Y_{r_C}$ according to the distribution of $X_{\varphi(r_C)}$ independently of all choices made so far. Then, for each child $v$ of $r_C$, one by one, we apply Lemma~\ref{lem:glue} with the role of $A$ played by the tuple of all variables $Y_u$ defined so far, except for $Y_{r_C}$, the role of $B_1$ played by $Y_{r_C}$, the role of $B_2$ played by $X_{\varphi(r_C)}$ and the role of $C$ played by $X_{\varphi(v)}$. We take $Y_v$ to be the variable $C'$ resulting from this. After exhausting all children of $r_C$, we continue in the same way on the children of each child of $r_C$ following the order of a standard ``breadth-first search.'' Once $Y_v$ has been defined for every vertex $v\in V(C)$, we then move on to another component, and so on, in such a way that variables corresponding to vertices in different components of $H$ are independent of one another.

The random variable $(Y_v:v\in V(H))$ satisfies \ref{eq:homSupportedGeneral} simply by construction. Let us show that it satisfies \ref{eq:same}. By Lemmas~\ref{lem:chain} and~\ref{lem:condIndep}, $\mathbb{H}(Y_v: v\in V(H))$ is equal to
\begin{equation}\label{eq:sumOverC}\sum_C\left(\mathbb{H}(Y_{r_C}) + \sum_{v\in V(C)\setminus\{r_C\}}\mathbb{H}(Y_v\mid Y_{p(v)})\right)\end{equation}
where the outer summation is over all components $C$ of $H$. Applying the Chain Rule again, we see that, for any non-root $v$,
\begin{equation}\label{eq:YCR}\mathbb{H}(Y_v\mid Y_{p(v)})=\mathbb{H}(Y_v, Y_{p(v)}) - \mathbb{H}(Y_{p(v)}) = \mathbb{H}(X_{\varphi(v)},X_{\varphi(p(v))})-\mathbb{H}(X_{\varphi(p(v))}).\end{equation} 
Since $H$ is a forest, every edge of $H$ is between some vertex $v$ and its parent. Also, we observe that each root vertex $r_C$ is the parent of exactly $d(r_C)$ vertices and each non-root vertex $v$ of $H$ is the parent of $d(v)-1$ vertices, where $d(v)$ denotes the total degree of $v$ (including edges of both colours). So, by substituting \eqref{eq:YCR} into \eqref{eq:sumOverC} and collecting terms, we see that 
\[\mathbb{H}(Y_v: v\in V(H))=\sum_{uv\in E(H)}\mathbb{H}(X_{\varphi(u)},X_{\varphi(v)}) - \sum_{v\in V(H)}(d(v)-1)\mathbb{H}(X_{\varphi(v)}).\]
Since $\varphi$ covers every edge of $P_{2k+1}^A$ exactly $e(H)/e(P_{2k+1}^A)$ times, the first summation in the above expression can be expressed as follows:
\[\sum_{uv\in E(H)}\mathbb{H}(X_{\varphi(u)},X_{\varphi(v)}) = (e(H)/e(P_{2k+1}^A))\sum_{i=0}^{2k}\mathbb{H}(X_i,X_{i+1}).\]
Now, let us analyze the second summation. We have 
\begin{align*}
\sum_{v\in V(H)}(d(v)-1)\mathbb{H}(X_{\varphi(v)}) &= \sum_{v\in V(H)}d(v)\mathbb{H}(X_{\varphi(v)})-\sum_{v\in V(H)}\mathbb{H}(X_{\varphi(v)})\\ &=\sum_{v\in V(H)}\sum_{\substack{e\in E(H)\\ v\in e}}\mathbb{H}(X_{\varphi(v)}) -\sum_{v\in V(H)}\mathbb{H}(X_{\varphi(v)}).
\end{align*}
Each of the two terms in the final expression above is nothing more than a linear combination of the values $\mathbb{H}(X_i)$ for $0\leq i\leq 2k+1$. We analyze these linear combinations separately. For the first one, for each $0\leq i\leq 2k+1$, the coefficient of $\mathbb{H}(X_i)$ is equal to the number of edges of $H$ that are mapped to the edges of $P_{2k+1}^A$ that are incident to vertex $i$. So, since $\varphi$ covers every edge exactly $e(H)/e(P_{2k+1}^A)$ times, the coefficient of $\mathbb{H}(X_i)$ in the first term is equal to $e(H)/e(P_{2k+1}^A)$ or $2(e(H)/e(P_{2k+1}^A))$, depending on whether vertex $i$ is a leaf of the path or not. On the other hand, since $\varphi$ covers every vertex exactly $e(H)/e(P_{2k+1}^A)$ times as well, the coefficient of $\mathbb{H}(X_i)$ in the second term is equal to $e(H)/e(P_{2k+1}^A)$ for all $0\leq i\leq 2k+1$. Putting all of this together and applying Proposition~\ref{prop:pathEntropy}, we get that \ref{eq:same} holds. 

Finally, applying \ref{eq:homSupportedGeneral}, \ref{eq:same} and Lemma~\ref{lem:unif} (in two different ways), we obtain
\begin{align*}\log_2(\hom(H,G)) \geq \mathbb{H}(Y_v: v\in V(H)) &= (e(H)/e(P_{2k+1}^A))\mathbb{H}(X_0,\dots,X_{2k+1})\\ &= (e(H)/e(P_{2k+1}^A))\log_2(\hom(P_{2k+1}^A,G)).\end{align*}
The inequality in the lemma now follows from taking 2 to the power of both sides and dividing by $n^{v(H)}$, which is equal to $n^{v(P_{2k+1}^A)(e(H)/e(P_{2k+1}^A))}$ because $\varphi$ covers ever vertex exactly $e(H)/e(P_{2k+1}^A)$ times. 
\end{proof}

\section{The Technical Construction}
\label{sec:construction}

The final step in our proof of Theorem~\ref{thm:paths} is to construct the edge-coloured forest $H_{2k+1}$ as in Lemma~\ref{lem:key}. To establish the lower bound \eqref{eq:PH}, we will apply Lemma~\ref{lem:entropy}. Therefore, in addition to the edge-coloured forest $H_{2k+1}$, we will need a homomorphism $\varphi_{2k+1}$ from $H_{2k+1}$ to $P_{2k+1}^A$ covering every edge and vertex the same number of times. Moreover, we want to construct $H_{2k+1}$ in such a way that the upper bound \eqref{eq:HP} can be proven via simple local counting arguments that are only slightly more complicated than those that were used for the warm up example in Section~\ref{sec:sketch}. 

Our first step will be to define, for each $k\geq1$, three sequences $\vec{x}=(x_0,\dots,x_{2k+1})$, $\vec{y}=(y_0,\dots,y_{2k})$ and $\vec{z}=(z_0,\dots,z_{2k+1})$. The quantities $x_i$ and $z_i$ will control the number of pendant edges and $2$-edge paths attached to the vertex $v_i$ during the construction of $H_{2k+1}$, while $y_i$ controls the number of disjoint ``copies'' of the edge $v_iv_{i+1}$ added. It may be useful for the reader to refer back to the construction of $H_5$ in Figure~\ref{fig:H5} to get a sense of what we mean by this.

\begin{defn}
\label{defn:sequences}
For $k\geq1$, define the vectors $\vec{x}=(x_0,\dots,x_{2k+1})$, $\vec{y}=(y_0,\dots,y_{2k})$ and $\vec{z}=(z_0,\dots,z_{2k+1})$ as follows. Each vector is symmetric, meaning $x_i=x_{2k+1-i}, y_{i}=y_{2k-i}$ and $z_i=z_{2k+1-i}$ for $i\in\{0,1,\ldots,k\}$; hence, it suffices to define their entries in the first half. The definitions depend on the parity of $k$, with small cases treated separately.\\

\noindent
If $k=1$, then define

\noindent
\begin{minipage}[t]{0.48\textwidth}
\begin{itemize}
    \item\label{eq:x3} \( x_0 := 0, \ x_1 := 4 \),
\end{itemize}

\begin{itemize}
    \item\label{eq:z3} \( z_0 = z_1 := 0 \).
\end{itemize}
\end{minipage}%
\hfill
\begin{minipage}[t]{0.48\textwidth}
\begin{itemize}
    \item\label{eq:y3} \( y_0 = y_1 := 2 \),
\end{itemize}
\end{minipage}
\\

\noindent
If $k=2$, then define  

\noindent
\begin{minipage}[t]{0.48\textwidth}
\begin{itemize}
    \item\label{eq:x5} $x_0 := 0,\ x_1 := 16,\ x_2 := 10$,
\end{itemize}
\begin{itemize}
    \item\label{eq:z5} $z_0 := 0,\ z_1 := 0,\ z_2 := 5$.
\end{itemize}
\end{minipage}%
\hfill
\begin{minipage}[t]{0.48\textwidth}
\begin{itemize}
    \item\label{eq:y5} $y_0 := 14,\ y_1 := 1,\ y_2 := 0$,
\end{itemize}
\end{minipage}
\\

\noindent
If $k=3$, then define

\noindent
\begin{minipage}[t]{0.48\textwidth}
\begin{itemize}
\item\label{eq:x7} $x_0:=0, \ x_1:=48,\ x_2:=40,\ x_3:=28$,
\end{itemize}
\begin{itemize}
\item\label{eq:y7} $y_0:=36,\ y_1:=4,\ y_2:=2,\ y_3:=0$.
\end{itemize}
\end{minipage}%
\hfill
\begin{minipage}[t]{0.48\textwidth}
\begin{itemize}
\item\label{eq:z7} $z_0:=0,\ z_1:=0,\ z_2:=0,\ z_3:=14$, 
\end{itemize}
\end{minipage}
\\

\noindent
If $k$ is odd and $k\geq 5$, for each integer $0\leq i\leq \left\lfloor\frac{k-5}{4}\right\rfloor$, define

\noindent
\begin{minipage}[t]{0.48\textwidth}
\begin{itemize}
    \item\label{eq:13x0} \( x_{4i} := \begin{cases}0 & \text{if } i = 0,\\ (k+1)(k^2 - i) & \text{otherwise} \end{cases} \),
\end{itemize}
\begin{itemize}
    \item\label{eq:13x1} \( x_{4i+1} := (k+1)(k^2 + k - i) \),
\end{itemize}
\begin{itemize}
    \item\label{eq:13x2} \( x_{4i+2} := (k+1)(k^2 + i) \),
\end{itemize}
\begin{itemize}
    \item\label{eq:13x3} \( x_{4i+3} := (k+1)(k^2 - k + i) \),
\end{itemize}
\begin{itemize}
    \item\label{eq:13z0} \( z_{4i} := (2k + 1)i \),
\end{itemize}
\begin{itemize}
    \item\label{eq:13z2} \( z_{4i+2} := 0 \),
\end{itemize}
\end{minipage}%
\hfill
\begin{minipage}[t]{0.48\textwidth}
\begin{itemize}
    \item\label{eq:13y0} \( y_{4i} := \begin{cases}k^2(k+1) & \text{if } i = 0,\\ i & \text{otherwise} \end{cases} \),
\end{itemize}
\begin{itemize}
    \item\label{eq:13y1} \( y_{4i+1} := 0 \),
\end{itemize}
\begin{itemize}
    \item\label{eq:13y2} \( y_{4i+2} := k - i \),
\end{itemize}
\begin{itemize}
    \item\label{eq:13y3} \( y_{4i+3} := 0 \),
\end{itemize}
\begin{itemize}
    \item\label{eq:13z1} \( z_{4i+1} := 0 \),
\end{itemize}
\begin{itemize}
    \item\label{eq:13z3} \( z_{4i+3} := (2k + 1)(k - i) \).
\end{itemize}
\end{minipage}
\\

\noindent
If $k\equiv 1\bmod 4$ and $k\geq 5$, then define \\
\noindent
\begin{minipage}[t]{0.48\textwidth}
\begin{itemize}
    \item\label{eq:1xk-1} \( x_{k-1} := \frac{4k^3 + 3k^2 + 1}{4} \),
\end{itemize}
\begin{itemize}
    \item\label{eq:1yk-1} \( y_{k-1} := \frac{k - 1}{4} \),
\end{itemize}
\begin{itemize}
    \item\label{eq:1zk-1} \( z_{k-1} := \frac{(k - 1)(2k + 1)}{4} \),
\end{itemize}
\end{minipage}%
\hfill
\begin{minipage}[t]{0.48\textwidth}
\begin{itemize}
    \item\label{eq:1xk} \( x_k := \frac{4k^3 + 7k^2 + 4k + 1}{4} \),
\end{itemize}
\begin{itemize}
    \item\label{eq:1yk} \( y_k := \frac{(k + 1)^2}{2} \),
\end{itemize}
\begin{itemize}
    \item\label{eq:1zk} \( z_k := 0 \).
\end{itemize}
\end{minipage}
\\

\noindent
If $k\equiv 3\bmod 4$ and $k\geq 7$, then we define 

\noindent
\begin{minipage}[t]{0.48\textwidth}
\begin{itemize}
    \item\label{eq:3xk-3} \( x_{k-3} := \frac{4k^3 + 3k^2 + 2k + 3}{4} \),
\end{itemize}
\begin{itemize}
    \item\label{eq:3xk-1} \( x_{k-1} := \frac{2k^3 + 3k^2 - 1}{2} \),
\end{itemize}
\begin{itemize}
    \item\label{eq:3yk-3} \( y_{k-3} := \frac{k - 3}{4} \),
\end{itemize}
\begin{itemize}
    \item\label{eq:3yk-1} \( y_{k-1} := \frac{k + 1}{2} \),
\end{itemize}
\begin{itemize}
    \item\label{eq:3zk-3} \( z_{k-3} := (2k + 1)\frac{k - 3}{4} \),
\end{itemize}
\begin{itemize}
    \item\label{eq:3zk-1} \( z_{k-1} := 0 \),
\end{itemize}
\end{minipage}%
\hfill
\begin{minipage}[t]{0.48\textwidth}
\begin{itemize}
    \item\label{eq:3xk-2} \( x_{k-2} := \frac{4k^3 + 7k^2 + 6k + 3}{4} \),
\end{itemize}
\begin{itemize}
    \item\label{eq:3xk} \( x_k := \frac{2k^3 + k^2 - 2k - 1}{2} \),
\end{itemize}
\begin{itemize}
    \item\label{eq:3yk-2} \( y_{k-2} := \frac{(k + 1)^2}{4} \),
\end{itemize}
\begin{itemize}
    \item\label{eq:3yk} \( y_k := 0 \),
\end{itemize}
\begin{itemize}
    \item\label{eq:3zk-2} \( z_{k-2} := 0 \),
\end{itemize}
\begin{itemize}
    \item\label{eq:3zk} \( z_k := \frac{2k^2 + 3k + 1}{2} \).
\end{itemize}
\end{minipage}
\\

\noindent
If $k$ is even and $k\geq 4$, for each integer $0\leq i\leq \left\lfloor\frac{k-4}{4}\right\rfloor$, define

\noindent
\begin{minipage}[t]{0.48\textwidth}
\begin{itemize}
    \item\label{eq:02x0} \( x_{4i} := \begin{cases}0 & \text{if } i = 0,\\ k(k^2 + k + i) & \text{otherwise} \end{cases} \),
\end{itemize}
\begin{itemize}
    \item\label{eq:02x1} \( x_{4i+1} := k((k+1)^2 - i) \),
\end{itemize}
\begin{itemize}
    \item\label{eq:02x2} \( x_{4i+2} := k(k^2 + k - i) \),
\end{itemize}
\begin{itemize}
    \item\label{eq:02x3} \( x_{4i+3} := k(k^2 + i) \),
\end{itemize}
\begin{itemize}
    \item\label{eq:02z0} \( z_{4i} := 0 \),
\end{itemize}
\begin{itemize}
    \item\label{eq:02z2} \( z_{4i+2} := (2k + 1)i \),
\end{itemize}
\end{minipage}%
\hfill
\begin{minipage}[t]{0.48\textwidth}
\begin{itemize}
    \item\label{eq:02y0} \( y_{4i} := \begin{cases}k^2(k+1) & \text{if } i = 0,\\ 0 & \text{otherwise} \end{cases} \),
\end{itemize}
\begin{itemize}
    \item\label{eq:02y1} \( y_{4i+1} := i \),
\end{itemize}
\begin{itemize}
    \item\label{eq:02y2} \( y_{4i+2} := 0 \),
\end{itemize}
\begin{itemize}
    \item\label{eq:02y3} \( y_{4i+3} := k - i \),
\end{itemize}
\begin{itemize}
    \item\label{eq:02z1} \( z_{4i+1} := 0 \),
\end{itemize}
\begin{itemize}
    \item\label{eq:02z3} \( z_{4i+3} := (2k + 1)(k - i) \).
\end{itemize}
\end{minipage}
\\

\noindent
If $k\equiv 0\bmod 4$, then we define

\noindent
\begin{minipage}[t]{0.48\textwidth}
\begin{itemize}
    \item\label{eq:mx00} \( x_{k} := k^2\left(k + \frac{5}{4}\right)  \),
\end{itemize}
\begin{itemize}
    \item\label{eq:mz00} \( z_{k} := 0 \).
\end{itemize}
\end{minipage}%
\hfill
\begin{minipage}[t]{0.48\textwidth}
\begin{itemize}
    \item\label{eq:my00} \( y_{k} := \frac{(k + 2)k}{2} \),
\end{itemize}
\end{minipage}
\\

\noindent
If $k\equiv 2\bmod 4$ and $k \geq 6$, then we define

\noindent
\begin{minipage}[t]{0.48\textwidth}
\begin{itemize}
    \item\label{eq:mx022} \( x_{k-2} := \frac{(4k^2 + 5k - 2)k}{4} \),
\end{itemize}
\begin{itemize}
    \item\label{eq:mx020} \( x_{k} := \frac{k^2(2k + 1)}{2} \),
\end{itemize}
\begin{itemize}
    \item\label{eq:my021} \( y_{k-1} := \frac{k}{2} \),
\end{itemize}
\begin{itemize}
    \item\label{eq:mz022} \( z_{k-2} := 0 \),
\end{itemize}
\begin{itemize}
    \item\label{eq:mz020} \( z_{k} := \frac{(2k + 1)k}{2} \).
\end{itemize}
\end{minipage}%
\hfill
\begin{minipage}[t]{0.48\textwidth}
\begin{itemize}
    \item\label{eq:mx021} \( x_{k-1} := \frac{(2k^2 + 3k + 2)k}{2} \),
\end{itemize}
\begin{itemize}
    \item\label{eq:my022} \( y_{k-2} := \frac{(k + 2)k}{4} \),
\end{itemize}
\begin{itemize}
    \item\label{eq:my020} \( y_k := 0 \),
\end{itemize}
\begin{itemize}
    \item\label{eq:mz021} \( z_{k-1} := 0 \),
\end{itemize}
\end{minipage}
\end{defn}

The following lemma can be easily observed. 
\begin{lem}
For all $k\geq1$, all of the entries of $\vec{x},\vec{y}$ and $\vec{z}$ are non-negative integers.
\end{lem}








For convenience, from here forward, we view the sequences $\vec{x},\vec{y},\vec{z}$ as being indexed by $\mathbb{Z}$, where every entry that was not defined in Definition~\ref{defn:sequences} is equal to zero. The next two lemmas will be used to show that our construction $H_{2k+1}$ admits a homomorphism to $P_{2k+1}^A$ that covers every vertex and edge the same number of times. These lemmas are straightforward but tedious to verify; we have included the proofs in an appendix at the end of the paper. 

\begin{lem}
\label{lem:vertexCovered}
For every $k\geq1$ and $0\leq j\leq k$, the following expressions evaluate to $k(k+1)^2(2k+1)$:
\begin{enumerate}[label=\emph{(v\arabic*)}, ref=(v\arabic*)]
\addtocounter{enumi}{-1}
    \item\label{eq:v0} $c(v_{2j}):=kx_{2j-1}+(k+1)x_{2j+1}+(k+1)y_{2j-1}+(k+1)y_{2j}+(k+1)z_{2j-2}+kz_{2j-1}+kz_{2j}+(k+1)z_{2j+1}$,
    \item\label{eq:v1} $c(v_{2j+1}):=(k+1)x_{2j}+kx_{2j+2}+(k+1)y_{2j}+(k+1)y_{2j+1}+(k+1)z_{2j}+kz_{2j+1}+kz_{2j+2}+(k+1)z_{2j+3}$.
\end{enumerate}
\end{lem}

\begin{lem}
\label{lem:edgeCovered}
For every $k\geq1$, the following expressions evaluate to $k(k+1)^2(2k+1)$:
\begin{enumerate}[label=\emph{(e\arabic*)}, ref=(e\arabic*)]
\addtocounter{enumi}{-1}
    \item\label{eq:e0} $c(v_{2j},v_{2j+1}):=(k+1)x_{2j} + (k+1)x_{2j+1} + (k+1)y_{2j} + (k+1)z_{2j} + (k+1)z_{2j+1}$ for $0\leq j\leq k$
    \item\label{eq:e1} $c(v_{2j+1},v_{2j+2}):=kx_{2j+1} + kx_{2j+2} + (k+1)y_{2j+1} + (k+1)z_{2j} + 2kz_{2j+1} + 2kz_{2j+2} + (k+1)z_{2j+3}$ for $j\leq k-1$.
\end{enumerate}
\end{lem}

We are now in position to prove Lemma~\ref{lem:key}. 

\begin{proof}[Proof of Lemma~\ref{lem:key}]
The edge-coloured forest $H_{2k+1}$ is defined in terms of the vectors $\vec{x}=(x_0,\dots,x_{2k+1})$, $\vec{y}=(y_0,\dots,y_{2k})$ and $\vec{z}=(z_0,\dots,z_{2k+1})$ in Definition~\ref{defn:sequences} as follows. 

We start with the alternating path $P_{2k+1}^A$ with vertices $v_0,\dots,v_{2k+1}$. Next, for each vertex $v_j$ for $0\leq j\leq 2k+1$, add two sets of vertices $X_j^R$ and $X_j^B$ such that $|X_j^R|=(k+1)x_j$, $|X_j^B|=kx_j$, every vertex of $X_j^R$ is adjacent to $v_j$ via a red edge and every vertex of $X_j^B$ is adjacent to $v_j$ via a blue edge. Now, for each $0\leq j\leq 2k$, add $(k+1)y_j$ isolated edges of the same colour as the edge $v_jv_{j+1}$; we let $(Y_j^-,Y_j^+)$ be a partition of the vertices added in this step so that each isolated edge has one endpoint in $Y_j^-$ and one in $Y_j^+$. Finally, for each $0\leq j\leq 2k+1$, add sets $Z^R_j$, $Z^{RB}_j$, $Z^B_j$ and $Z^{BB}_j$ of vertices where $|Z^R_j|=|Z^{RB}_j|=(k+1)z_j$, $|Z^B_j|=|Z^{BB}_j|=kz_j$, every vertex of $Z^R_j$ is adjacent to $v_j$ via a red edge, every vertex of $Z^B_j$ is adjacent to $v_j$ via a blue edge, and there are blue perfect matchings between $Z^R_j$ and $Z^{RB}_j$ and between $Z^B_j$ and $Z^{BB}_j$. This completes the definition of $H_{2k+1}$. 

Now, let $\varphi$ be a homomorphism from $H_{2k+1}$ to $P_{2k+1}^A$ with the property that $\varphi(Y_j^-)=\{v_j\}$ and $\varphi(Y_j^+)=\{v_{j+1}\}$ for $0\leq j\leq 2k$. Note that, since every vertex of $P_{2k+1}^A$ is incident to at most one edge of each colour, this homomorphism is unique. Also, it is not hard to check that each vertex $v_j$ is covered exactly $c(v_j)+1$ times by $\varphi$ and each edge $v_jv_{j+1}$ is covered exactly $c(v_j,v_{j+1})+1$ times, where these quantities are defined in Lemmas~\ref{lem:vertexCovered} and~\ref{lem:edgeCovered}. By Lemmas~\ref{lem:vertexCovered} and~\ref{lem:edgeCovered}, we have that every vertex and edge is covered exactly $k(k+1)^2(2k+1)+1$ times. Therefore, \eqref{eq:PH} holds by Lemma~\ref{lem:entropy}. 

Our final task is to prove \eqref{eq:HP}. Let $G$ be any edge-coloured graph. Let $H_{2k+1}'$ be the edge-coloured graph obtained from $H_{2k+1}$ by deleting all vertices of $\bigcup_{j=0}^{2k+1}(X_j^R\cup X_j^B)$. Since $H_{2k+1}'$ is a subgraph of $H_{2k+1}$, the restriction of any homomorphism from $H_{2k+1}$ to $G$ to the set $V(H_{2k+1}')$ is a homomorphism from $H_{2k+1}'$ to $G$. For each $f\in\Hom(H_{2k+1}',G)$, let $\hom(H_{2k+1},G;f)$ be the number of homomorphisms from $H_{2k+1}$ to $G$ whose restriction to $V(H_{2k+1}')$ is $f$. Then
\begin{align*}\hom(H_{2k+1},G) &= \sum_{f\in \Hom(H_{2k+1}',G)}\hom(H_{2k+1},G;f)\\
&= \sum_{f\in \Hom(H_{2k+1}',G)}\prod_{j=0}^{2k+1}d_R(f(v_j))^{(k+1)x_j}d_B(f(v_j))^{kx_j}\\
&\leq \sum_{f\in \Hom(H_{2k+1}',G)}\prod_{j=0}^{2k+1}\left((k+1)^{k+1}k^k(2k+1)^{-2k-1}(n-1)^{2k+1}\right)^{x_j}\\
&= \prod_{j=0}^{2k+1}\left((k+1)^{k+1}k^k(2k+1)^{-2k-1}(n-1)^{2k+1}\right)^{x_j} \hom(H_{2k+1}',G)
\end{align*}
where the inequality in the penultimate step is a simple consequence of the fact that $d_R(u)+d_B(u)\leq n-1$ for every vertex $u$ of $G$; see the footnote\footref{calculus}. Thus, we have
\begin{equation}\label{eq:stripLeaves}t(H_{2k+1},G)\leq \left(k^{\frac{k}{2k+1}}(k+1)^{\frac{k+1}{2k+1}}(2k+1)^{-1}\right)^{e(H_{2k+1})-e(H_{2k+1}')} t(H_{2k+1}',G).\end{equation}

Next, let $H_{2k+1}''$ be obtained from $H_{2k+1}'$ by deleting all vertices of $\bigcup_{j=0}^{2k+1}(Z_j^R\cup Z_j^B\cup Z_j^{RB}\cup Z_j^{BB})$ for all $0\leq j\leq 2k+1$. Analogously to before, for each $f\in \Hom(H_{2k+1}'',G)$, let $\hom(H_{2k+1}',G;f)$ be the number of homomorphisms from $H_{2k+1}'$ to $G$ whose restriction to $V(H_{2k+1}'')$ is $f$. We let $d_{RB}(v)$ be the number of walks of length two in $G$ starting with a vertex $v$ in which the edges are red followed by blue, and define $d_{BB}(v)$ analogously but where both edges are blue. Then
\begin{align*}\hom(H_{2k+1}',G) &= \sum_{f\in \Hom(H_{2k+1}'',G)}\hom(H_{2k+1}',G;f)\\
&= \sum_{f\in \Hom(H_{2k+1}'',G)}\prod_{j=0}^{2k+1}d_{RB}(f(v_j))^{(k+1)z_j}d_{BB}(f(v_j))^{kz_j}.
\end{align*}
A simple double-counting argument tells us that $d_{RB}(v)+d_{BB}(v) \leq 2e_B(G) - d_B(v)$ where $e_B(G)$ denotes the number of blue edges of $G$. Thus, using the footnote\footref{calculus} again, $\hom(H_{2k+1}',G)$ is at most 
\[\prod_{j=0}^{2k+1}\left((k+1)^{k+1}k^k(2k+1)^{-2k-1}(2e_B(G))^{2k+1}\right)^{z_j} \hom(H_{2k+1}'',G).
\]
Let $H_{2k+1}'''$ be obtained from $H_{2k+1}''$ by adding $\sum_{j=0}^{2k+1}(2k+1)z_j$ isolated blue edges. Then, as an isolated blue edge has exactly $2e_B(G)$ homomorphisms to $G$, we have that $\hom(H_{2k+1}''',G) = (2e_B(G))^{\sum_{j=0}^{2k+1}(2k+1)z_j}\hom(H_{2k+1}'',G)$. Thus, the upper bound that we proved on $\hom(H_{2k+1}',G)$ can be rewritten as 
\[\hom(H_{2k+1}',G)\leq \prod_{j=0}^{2k+1}\left((k+1)^{k+1}k^k(2k+1)^{-2k-1}\right)^{z_j} \hom(H_{2k+1}''',G).\]
This easily implies that 
\begin{equation}\label{eq:cut2Paths}t(H_{2k+1}',G)\leq \left(k^{\frac{k}{2k+1}}(k+1)^{\frac{k+1}{2k+1}}(2k+1)^{-1}\right)^{e(H_{2k+1}')-e(H_{2k+1}''')} t(H_{2k+1}''',G).\end{equation}

Finally, we observe that $H_{2k+1}'''$ consists of the union of $P_{2k+1}^A$ and a bunch of isolated edges of each colour. We claim that the number of isolated red edges and the number of isolated blue edges are in a ratio of $k+1$ to $k$. To this end, we start by observing that, since $\varphi$ covers every edge of $P_{2k+1}^A$ the same number of times, and the number of red vs. blue edges of $P_{2k+1}^A$ are in a $k+1$ to $k$ ratio, the number of red vs. blue edges of $H_{2k+1}$ are in a $k+1$ to $k$ ratio. When changing from $H_{2k+1}$ to $H_{2k+1}'$, the number of red or blue edges deleted were also in a ratio of $k+1$ to $k$. Also, when transforming from $H_{2k+1}'$ to $H_{2k+1}'''$, the number of red edges lost is $\sum_{j=0}^{2k+1}|Z_j^R|=(k+1)\sum_{j=1}^{2k+1}z_j$ and the number of blue edges lost is $\sum_{j=0}^{2k+1}|Z_j^B|=k\sum_{j=1}^{2k+1}z_j$; so, they are again in a ratio of $k+1$ to $k$. Thus, the number of red edges vs. blue edges of $H_{2k+1}'''$ is in a ratio of $k+1$ to $k$ which, since $P_{2k+1}^A$ has this property, means that the number of isolated edges of $H_{2k+1}'''$ is in this ratio. Thus, if $m$ denotes the number of isolated red edges of $H_{2k+1}'''$, then, by the footnote\footref{calculus},
\begin{align*}\hom(H_{2k+1}''',G)& = (2e_R(G))^m (2e_B(G))^{\frac{k}{k+1}m}\hom(P_{2k+1}^A,G)\\
 & \leq \left((k+1)^{k+1}k^k(2k+1)^{-2k-1}(n(n-1))^{2k+1}\right)^{\frac{m}{k+1}}\hom(P_{2k+1}^A,G)\end{align*}
 which gives us that
 \begin{equation}\label{eq:isolatedEdges}t(H_{2k+1}''',G)\leq \left(k^{\frac{k}{2k+1}}(k+1)^{\frac{k+1}{2k+1}}(2k+1)^{-1}\right)^{e(H_{2k+1}''')-e(P_{2k+1}^A)} t(P_{2k+1}^A,G).\end{equation}
 The bound in \eqref{eq:HP} now follows by combining \eqref{eq:stripLeaves}, \eqref{eq:cut2Paths} and \eqref{eq:isolatedEdges}. This completes the proof of the lemma.
\end{proof}

In this section, we have described the vectors $\vec{x},\vec{y}$ and $\vec{z}$ and used them to construct the forest $H_{2k+1}$ satisfying Lemma~\ref{lem:key}. However, we have not addressed the way in which these arcane vectors were found. For any given value of $k$, the search for the vectors can be performed by running a simple linear program. That is, we set each of the entries of $\vec{x},\vec{y}$ and $\vec{z}$ to be a non-negative variable  and write down a list of constraints which implies that each vertex and edge of $P^A_{2k+1}$ is covered exactly once by the natural homomorphism of $H_{2k+1}$ to $P_{2k+1}^A$ and the sum of the variables is equal to one. Then any feasible solution to these constraints corresponds to a valid choice of the variables, except for the fact that our variables may be rational numbers as opposed to integers. However, scaling all of the variables by a multiple of the lowest common multiple of all of their denominators will give us integer values. The way that we found the vectors in Definition~\ref{defn:sequences} was by solving this linear program for various paths of length at most 40, or so, and then trying to find a pattern that works for all lengths. The fact that the solution is far from unique (see, e.g., Figure~\ref{fig:H3}) made it somewhat difficult to find a nice pattern; we needed to add various additional constraints to guide the linear programming solver to a solution that could be generalized to all odd path lengths.

\section{Conclusion}
\label{sec:conclusion}

While our main focus was on alternating paths, it seems clear that the approach taken in this paper can be strengthened further and applied to more general classes of edge-coloured trees and, more generally, forests. We did not pursue this direction as the case of alternating paths seemed most natural for us, and the construction was already quite involved. It would, however, be interesting to explore the class of edge-coloured trees $T$ with the property that $t(T,G)$ is bounded above by $e_R(T)^{e_R(T)}e_B(T)^{e_B(T)}(e_R(T)+e_B(T))^{-e_R(T)-e_B(T)}$ in greater generality. 

It would also be interesting to study such inequalities for edge-coloured graphs containing cycles. A lovely problem posed in~\cite[Problem~9.1]{Basit+25+} is to solve the semi-inducibility problem for alternating cycles of length $2\bmod 4$; the case of alternating cycles of length $0\bmod4$ is settled in~\cite[Theorem~1.5]{Basit+25+}. A flag algebra calculation shows that the answer for the alternating $6$-cycle is $(1/2)^6$.\footnote{Recall that we are working in the language of homomorphism density, not ``copies,'' and so all quantities that we discuss differ from those of~\cite{Basit+25+} by a constant factor depending on the number of symmetries of the cycle.} However, it would be challenging to extend the flag algebra approach to work for the $10$-cycle, and it is hard to imagine that it would be possible to generalize to arbitrarily long cycles. To reiterate~\cite[Problem~9.1]{Basit+25+}, it would be very interesting to settle the semi-inducubility problem for the alternating cycle of length $4k+2$ for all $k\geq1$. 

While the entropy method worked particularly well for alternating paths, there are other approaches which may fare better for different edge-coloured graphs. One example is the ``algebraic expansion'' trick that is frequently used in extremal problems in the area of graph limits; see, e.g.,~\cite{BehagueMorrisonNoel23+,Lovasz11,KimLee24,GrzesikLeeLidickyVolec22}. For example, this idea can be combined with known results on Sidorenko's Conjecture and weakly norming graphs to obtain short proofs of Theorem~\ref{thm:paths} in the case $k=1$ and a qualitative version of~\cite[Theorem~1.9]{Basit+25+}, as we show now. 

\begin{proof}[Alternative proof of Theorem~\ref{thm:paths} for $k=1$]
Let $G$ be an edge-coloured graph on $n$ vertices. It is clear that the maximum of $t(P_3^A,G)$ over all such $G$ is attained when $G$ is an edge-coloured clique, and so we can assume this. Let $R$ and $B$ be the $n$-vertex graphs consisting of the red and blue edges of $G$, respectively. Let $A_R$ and $A_B$ be their adjacency matrices. Then, since $A_R+A_B$ is equal to the all ones matrix minus the identity, we have that $\hom(P_3^A,G) + \hom(P_3,R)$ is equal to
\begin{align*}
\sum_{u,v,x,y\in V(G)}A_R(u,v)A_B(v,x)A_R(x,y) &+\sum_{u,v,x,y\in V(G)}A_R(u,v)A_R(v,x)A_R(x,y)\\
&= \sum_{u,v,x,y\in V(G)}A_R(u,v)(A_B(v,x)+A_R(v,x))A_R(x,y)\\
&= \sum_{u,v,x,y\in V(G)}A_R(u,v)A_R(x,y) - O(n^3)\\
&= \hom(K_2,R)^2- O(n^3).
\end{align*}
Therefore,
\[t(P_3^A,G) + t(P_3,R)=t(K_2,R)^2 -o(1)\]
and so
\[t(P_3^A,G) \leq t(K_2,R)^2-t(P_3,R)\]
which, since $P_3$ satisfies Sidorenko's Conjecture~\cite{Sidorenko93}, is at most $t(K_2,R)^2-t(K_2,R)^3$. The maximum of this over $0\leq t(K_2,R)\leq 1$ is precisely $(1/3)(2/3)^2$, which gives us the bound in Theorem~\ref{thm:paths} in the case $k=1$. 
\end{proof}

\begin{thm}[See Basit et al.~{\cite[Theorem~1.9]{Basit+25+}}]
If $H$ is the edge-coloured $4$-cycle with exactly one blue edge, then $t(H,G)\leq (1/4)(3/4)^3$ for every edge-coloured graph $G$. 
\end{thm}

\begin{proof}
Again, let $G$ be an edge-coloured clique on $n$ vertices, and let $R$ and $B$ be the $n$-vertex graphs consisting of the red and blue edges of this colouring, respectively. Then, analogously to the previous proof, we have
\[t(H,G) + t(C_4,R)=t(P_3,R) -o(1)\]
and so
\[t(H,G) \leq t(P_3,R) -t(C_4,R).\]
The graph $C_4$ is ``weakly norming'' (see~\cite{Hatami10,ConlonLee17,ConlonLee24}). As a result, the inequality $t(C_4,R)\geq t(P_3,R)^{4/3}$ holds for any graph $R$; see~\cite{ConlonLee24} for more background on these types of ``domination inequalities.'' Thus, we have
\[t(H,G) \leq t(P_3,R) -t(P_3,R)^{4/3}.\]
The result now follows by maximizing this upper bound over $0\leq t(P_3,R)\leq 1$. 
\end{proof}

While the use of algebraic expansion and domination inequalities may be powerful in some restricted cases, the expressions that one obtains when there are multiple edges of both colours are often very complicated and challenging to analyze. For example, we were unable to extend this approach to prove Theorem~\ref{thm:paths} for any value of $k$ other than $k=1$. We even tried incorporating the ``Schur convexity'' ideas of~\cite{KimLee24} into this approach, with no avail. 

A particularly adventurous direction for future work could be to study the semi-inducibility problem for graphs with more than two edge colours. The method described in this paper can be adapted to this setting. For example, suppose that $T$ is a tree in which there are $r$ colours, every non-leaf vertex is incident with precisely $r$ edges, one of each colour, and the number of edges of each colour is the same. Letting $H$ be the tree obtained by adding a pendant edge of every colour incident to every non-leaf vertex, plus one isolated vertex, one can get inequalities relating $T$ and $H$ analogous to Lemma~\ref{lem:key}. This immediatly proves that the the answer to the semi-induciblity for $T$ is $(1/r)^{e(T)}$, which can be seen as a multicolour analogue of the result of~\cite[Theorem~1.3]{Basit+25+} on alternating paths of even length. It would be interesting to consider trees of this type without the constraint that there are the same number of edges of each colour. The smallest such tree is depicted in Figure~\ref{fig:3col} below. It would also be interesting to solve the semi-inducibility problem for alternating paths of arbitrary length with more than two colours. 

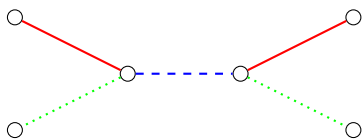
\begin{figure}[htbp]
\begin{center}
\begin{tikzpicture}[scale=1.5, every node/.style={draw, circle, fill=white, inner sep=2pt}]

\node (v0) at (0, 0.5) {};
\node (v1) at (1, 0) {};
\node (v2) at (2, 0) {};
\node (v3) at (3, 0.5) {};
\node (x11) at (0,-0.5){};
\node (x21) at (3,-0.5){};

\draw[red, thick] (v0) -- (v1); 
\draw[blue, dashed, thick] (v1) -- (v2); 
\draw[red, thick] (v2) -- (v3); 

\draw[green, dotted, thick] (v1) -- (x11); 
\draw[green, dotted, thick] (v2) -- (x21); 
\end{tikzpicture}
\end{center}
    \caption{A tree with edges of three colours.}
    \label{fig:3col}
\end{figure}

\begin{ack}
We learned about the semi-inducibility problem through remotely attending a talk given by Bertille Granet as part of the \emph{Connections Workshop: Extremal Combinatorics} held at SL Math in February, 2025. We thank SL Math for making the talks accessible for remote viewing and the third author thanks Bertille for helpful discussions over e-mail. 
\end{ack}

\bibliographystyle{plain}
\bibliography{SemiInducibility}

\begin{thebibliography}{10}

\bibitem{AlonSpencer}
N.~Alon and J.~H. Spencer.
\newblock {\em The probabilistic method}.
\newblock Wiley Series in Discrete Mathematics and Optimization. John Wiley \& Sons, Inc., Hoboken, NJ, fourth edition, 2016.

\bibitem{Basit+25+}
A.~Basit, B.~Granet, D.~Horsley, A.~K\"undgen, and K.~Staden.
\newblock The semi-inducibility problem.
\newblock E-print arXiv:2501.09842v1, 2025.

\bibitem{BehagueCrudeleNoelSimbaqueba+25+}
N.~Behague, G.~Crudele, J.~A. Noel, and L.~Simbaqueba.
\newblock {S}idorenko-type inequalities for pairs of trees.
\newblock E-print arXiv:2305.16542v2, 2025.

\bibitem{BehagueMorrisonNoel23+}
N.~Behague, N.~Morrison, and J.~A. Noel.
\newblock Common pairs of graphs.
\newblock E-print arXiv:2208.02045v3, Accepted to \emph{Combin. Probab. Comput.}, 2023.

\bibitem{BehagueMorrisonNoel24}
N.~Behague, N.~Morrison, and J.~A. Noel.
\newblock Off-diagonal commonality of graphs via entropy.
\newblock {\em SIAM J. Discrete Math.}, 38(3):2335--2360, 2024.

\bibitem{BlekhermanRaymond22}
G.~Blekherman and A.~Raymond.
\newblock A path forward: tropicalization in extremal combinatorics.
\newblock {\em Adv. Math.}, 407:Paper No. 108561, 68, 2022.

\bibitem{BlekhermanRaymond23}
G.~Blekherman and A.~Raymond.
\newblock A new proof of the {E}rd{\H o}s-{S}imonovits conjecture on walks.
\newblock {\em Graphs Combin.}, 39(3):Paper No. 53, 8, 2023.

\bibitem{ChaoYu24}
T.-W. Chao and H.-H.~H. Yu.
\newblock When entropy meets {T}ur\'an: new proofs and hypergraph {T}ur\'an results.
\newblock E-print arXiv:2412.08075v2, 2024.

\bibitem{ConlonKimLeeLee18}
D.~Conlon, J.~H. Kim, C.~Lee, and J.~Lee.
\newblock Some advances on {S}idorenko's conjecture.
\newblock {\em J. Lond. Math. Soc. (2)}, 98(3):593--608, 2018.

\bibitem{ConlonLee17}
D.~Conlon and J.~Lee.
\newblock Finite reflection groups and graph norms.
\newblock {\em Adv. Math.}, 315:130--165, 2017.

\bibitem{ConlonLee24}
D.~Conlon and J.~Lee.
\newblock Domination inequalities and dominating graphs.
\newblock {\em Math. Proc. Cambridge Philos. Soc.}, 177(1):167--184, 2024.

\bibitem{GrzesikLeeLidickyVolec22}
A.~Grzesik, J.~Lee, B.~Lidick\'{y}, and J.~Volec.
\newblock On tripartite common graphs.
\newblock {\em Combin. Probab. Comput.}, 31(5):907--923, 2022.

\bibitem{Hatami10}
H.~Hatami.
\newblock Graph norms and {S}idorenko's conjecture.
\newblock {\em Israel J. Math.}, 175:125--150, 2010.

\bibitem{KimLee24}
J.~S. Kim and J.~Lee.
\newblock Extended commonality of paths and cycles via {S}chur convexity.
\newblock {\em J. Combin. Theory Ser. B}, 166:109--122, 2024.

\bibitem{KoppartyRossman11}
S.~Kopparty and B.~Rossman.
\newblock The homomorphism domination exponent.
\newblock {\em European J. Combin.}, 32(7):1097--1114, 2011.

\bibitem{Lee21}
J.~Lee.
\newblock On some graph densities in locally dense graphs.
\newblock {\em Random Structures Algorithms}, 58(2):322--344, 2021.

\bibitem{Lovasz11}
L.~Lov\'{a}sz.
\newblock Subgraph densities in signed graphons and the local {S}imonovits-{S}idorenko conjecture.
\newblock {\em Electron. J. Combin.}, 18(1):Paper 127, 21, 2011.

\bibitem{Sidorenko93}
A.~F. Sidorenko.
\newblock A correlation inequality for bipartite graphs.
\newblock {\em Graphs Combin.}, 9(2):201--204, 1993.

\bibitem{Szegedy15}
B.~Szegedy.
\newblock An information theoretic approach to {S}idorenko's conjecture.
\newblock E-print arXiv:1406.6738v3, 2015.

\end{thebibliography}

\appendix

\section{Verification of the Covering Conditions}

The purpose of this appendix is to carry out the calculations required to prove Lemmas~\ref{lem:vertexCovered} and~\ref{lem:edgeCovered}; we start with the former. 

\begin{proof}[Proof of Lemma~\ref{lem:vertexCovered}]
Note that, by the symmetry of the construction in Definition~\ref{defn:sequences}, we need only consider vertices with index between $0$ and $k$. The proof is divided into cases based on the parity of $k$, handling small values separately.

\begin{case}
$k=1$. 
\end{case}
We have
\[c(v_0)=0+2\cdot 4+0+2 \cdot 2+ 0+0+0+0=12=k(k+1)^2(2k+1),\]
verifying \ref{eq:v0} and we have
\begin{align*}c(v_1) = 0+1\cdot 4+2 \cdot 2+2\cdot 2+0+0+0+0=12=k(k+1)^2(2k+1)\end{align*}
verifying \ref{eq:v1} in this case.

\begin{case}
$k=2$
\end{case}
We have 
\[c(v_0)=0+3\cdot 16+0+3\cdot 14+0+0+0+0=90=k(k+1)^2(2k+1)\]
and
\[c(v_2)=2\cdot16+3\cdot 10+3\cdot 1+0+0+0+2 \cdot 5 +3 \cdot 5
=90=k(k+1)^2(2k+1),\]
verifying \ref{eq:v0} and we have
\[c(v_1)=0+2\cdot 10+3\cdot 14+3\cdot 1+0+0+2\cdot 5 + 3 \cdot 5=90=k(k+1)^2(2k+1),\]
verifying \ref{eq:v1} in this case.

\begin{case}
$k=3$ 
\end{case}
We have
\[c(v_0)= 0+4\cdot 48+0+4+0+0+0+0\cdot 36=336=k(k+1)^2(2k+1)\]
and
\[c(v_2)=3\cdot 48+4\cdot 28+4\cdot 4+4 \cdot 2+0+0+0+4\cdot 14=336=k(k+1)^2(2k+1),\]
verifying \ref{eq:v0} in this case. Further, we have
\[c(v_1)=0+3\cdot 40+4\cdot 36+ 4\cdot 4+0+0+4\cdot 14=336=k(k+1)^2(2k+1)\]
and 
\[c(v_3)=4\cdot 40+3\cdot 28+4\cdot 2+0+0+3\cdot 14 + 3\cdot 14+0 =336=k(k+1)^2(2k+1)\]
verifying \ref{eq:v1} in this case. 
\begin{case}
$k$ is odd and $k\geq 5$. 
\end{case}

We start by proving \ref{eq:v0}. First, consider the case $j=0$. In this case,  we have
\[c(v_0)= 0+(k+1)^3k+0 + k^2(k+1)^2+0+0+0+0=k(k+1)^2(2k+1).\]
Next, suppose that $2j=4i$ for some $1\leq i\leq \left\lfloor\frac{k-5}{4}\right\rfloor$. In this case,
\begin{align*}
c(v_{2j}) &= k(k+1)(k^2-k+i-1) + (k+1)^2(k^2+k-i)+0+(k+1)i+0\\
&+k(2k+1)(k-i+1)+k(2k+1)i+0 = k(k+1)^2(2k+1).
\end{align*}
Now, suppose that $2j=4i+2$ for some $0\leq i\leq \left\lfloor\frac{k-5}{4}\right\rfloor$. We have 
\begin{align*}
c(v_{2j}) &=  k(k+1)(k^2+k-i) + (k+1)^2(k^2-k+i)+0+(k+1)(k-i)\\
&+(k+1)(2k+1)i+0+0+(k+1)(2k+1)(k-i)= k(k+1)^2(2k+1).
\end{align*}
To complete the proof of \ref{eq:v0}, it suffices to consider vertices of the form $v_{2j}$ where $4\left\lfloor \frac{k-1}{4}\right\rfloor\leq 2j\leq k$. 
First, consider the case $k\equiv 1\bmod 4$ and $k\geq5$. In this case, the only remaining index to check is $2j=k-1$; we have
\begin{align*}
c(v_{k-1}) =& k(k+1)\left(k^2-k+\frac{k-5}{4}\right)+(k+1)\frac{4k^3+7k^2+4k+1}{4}+0+(k+1)\frac{k-1}{4}+0\\
+&k(2k+1)\left(k-\frac{k-5}{4}\right)+k\left(\frac{(k-1)(2k+1)}{4}\right)+0= k(k+1)^2(2k+1).
\end{align*}
Next, consider the case $k\equiv 3\bmod 4$ and $k\geq7$. In this case, the remaining indices to check are $2j=k-3$ and $2j=k-1$; we have
\begin{align*}c(v_{k-3})&=k (k+1)\left(k^2-k+\frac{k-7}{4}\right)+(k+1)\frac{4k^3+7k^2+6k+3}{4}+0+(k+1)\frac{k-3}{4}+0\\
&+k(2k+1)\left(k-\frac{k-7}{4}\right)+k(2k+1)\frac{k-3}{4}+0=k(k+1)^2(2k+1)
\end{align*}
and
\begin{align*}
c(v_{k-1})&=k\frac{4k^3+7k^2+6k+3}{4}+(k+1)\frac{2k^3+k^2-2k-1}{2}+(k+1)\frac{(k+1)^2}{4}+(k+1)\frac{k+1}{2}\\
&+(k+1)(2k+1)\frac{k-3}{4}+0+0+(k+1)\frac{2k^2+3k+1}{2}=k(k+1)^2(2k+1).
\end{align*}
This completes the proof of \ref{eq:v0} in the case that $k$ is odd and $k\geq 5$. Next, we focus on \ref{eq:v1}. Suppose first that $j=0$. Then
\[c(v_{1})=0+k(k+1)k^2+(k+1)k^2(k+1)+0+0+0+k(2k+1)k+k(2k+1)=k(k+1)^2(2k+1).\]
Next, suppose that $2j+1=4i+1$ for some $1\leq i\leq \left\lfloor\frac{k-5}{4}\right\rfloor$. In this case,
\begin{align*}
c(v_{2j+1}) =& (k+1)^2(k^2-i)+k(k+1)(k^2+i)+(k+1)i+0+(k+1)(2k+1)i\\
+&0+0+(k+1)(2k+1)(k-i)= k(k+1)^2(2k+1).
\end{align*}
Now, suppose that $2j+1=4i+3$ for some $0\leq i\leq \left\lfloor\frac{k-5}{4}\right\rfloor$. We have 
\begin{align*}
c(v_{2j+1}) =& (k+1)(k+1)(k^2+i)+k(k+1)(k^2-i-1)+(k+1)(k-i)+0+0  \\
+&k(2k+1)(k-i)+k(2k+1)(i+1)+0= k(k+1)^2(2k+1).
\end{align*}
To complete the proof of \ref{eq:v1}, it suffices to consider vertices of the form $v_{2j+1}$ where $4\left\lfloor \frac{k-5}{4}\right\rfloor+3< 2j+1\leq k$. 
First, suppose that $k\equiv1\bmod 4$ and $k\geq5$. We only need to consider the case $2j+1=k$.
\begin{align*}
c(v_{k}) =&(k+1)\left(\frac{4k^3+3k^2+1}{4}\right)+k\frac{4k^3+7k^2+4k+1}{4}+(k+1)\frac{k-1}{4} +(k+1)\frac{(k+1)^2}{2}\\
+&(k+1)\frac{(k-1)(2k+1)}{4}+0+0+(k+1)\frac{(k-1)(2k+1)}{4}= k(k+1)^2(2k+1).
\end{align*}
Now, suppose that $k\equiv3\bmod 4$ and $k\geq7$. We need to consider the cases $2j+1=k-2$ and $2j+1=k$. We have
\begin{align*}
c(v_{k-2}) &=(k+1)\frac{4k^3+3k^2+2k+3}{4}+k\frac{2k^3+3k^2-1}{2}+(k+1)\frac{k-3}{4}+(k+1)\frac{(k+1)^2}{4}\\
&+(k+1)(2k+1)\frac{k-3}{4}+0+0+(k+1)\frac{2k^2+3k+1}{2}=k(k+1)^2(2k+1),
\end{align*}
and
\begin{align*}
c(v_{k}) &=(k+1)\frac{2k^3+3k^2-1}{2}+k\frac{2k^3+k^2-2k-1}{2}+(k+1)\frac{k+1}{2}+0+0\\
&+k\frac{2k^2+3k+1}{2}+k\frac{2k^2+3k+1}{2}+0=k(k+1)^2(2k+1).
\end{align*}

\begin{case}
$k$ is even and $k\geq 4$.
\end{case}
We start by proving \ref{eq:v0}. First, consider the case $j=0$. In this case,  we have
\[c(v_0)= 0+(k+1)k(k+1)^2+0+(k+1)k^2(k+1)+0+0+0+0=k(k+1)^2(2k+1).\]
Next, suppose that $2j=4i$ for some $1\leq i\leq \left\lfloor\frac{k-4}{4}\right\rfloor$. In this case,
\begin{align*}
c(v_{2j}) &= k^2(k^2+i-1)+(k+1)k((k+1)^2-i)+(k+1)(k+1-i)+0\\
&+(k+1)(2k+1)(i-1)+k(2k+1)(k+1-i)+0+0=k(k+1)^2(2k+1).
\end{align*}
Now, suppose that $2j=4i+2$ for some $0\leq i\leq \left\lfloor\frac{k-4}{4}\right\rfloor$. We have 
\begin{align*}
c(v_{2j}) &=k^2((k+1)^2-i)+(k+1)k(k^2+i)+(k+1)i+0+0+0\\
&+k(2k+1)i+(k+1)(2k+1)(k-i)=k(k+1)^2(2k+1).
\end{align*}
To complete the proof of \ref{eq:v0}, it suffices to consider vertices of the form $v_{2j}$ where $4\left\lfloor \frac{k-4}{4}\right\rfloor+2< 2j\leq k$. First, suppose that $k\equiv0\bmod 4$ and $k\geq4$. We only need to consider the case $2j=k$.
\begin{align*}
c(v_{k}) &= k^2\left(k^2+\frac{k}{4}-1\right)+(k+1)k^2\left(k+\frac{5}{4}\right)+(k+1)\left(\frac{3k}{4}+1\right)+(k+1)\frac{1}{2}(k+2)k\\
&+(k+1)(2k+1)\left(\frac{k}{4}-1\right)+k(2k+1)\left(\frac{3k}{4}+1\right)+0+0=k(k+1)^2(2k+1).
\end{align*}
Now, suppose that $k\equiv2\bmod 4$ and $k\geq6$. We need to consider the cases $2j=k-2$ and $2j=k$. We have
\begin{align*}
c(v_{k-2}) &=k^2\left(k^2+\frac{k}{4}-\frac{3}{2}\right)+(k+1)\frac{(2k^2+3k+2)k}{2}+(k+1)\left(\frac{3k}{4}+\frac{3}{2}\right)+(k+1)\frac{(k+2)k}{4}\\
&+(k+1)(2k+1)\left(\frac{k}{4}-\frac{3}{2}\right)+k(2k+1)\left(\frac{3k}{4}+\frac{3}{2}\right)+0+0=k(k+1)^2(2k+1) 
\end{align*}
and 
\begin{align*}
c(v_{k}) &=k\frac{k(2k^2+3k+2)}{2}+(k+1)\frac{k^2(2k+1)}{2}+(k+1)\frac{k}{2}+0+0+0\\
&+k\frac{k(2k+1)}{2}+(k+1)\frac{(2k+1)k}{2}=k(k+1)^2(2k+1).
\end{align*}
This completes the proof of \ref{eq:v0}. Next, we focus on \ref{eq:v1}. Suppose first that $j=0$. Then
\[c(v_{1})=0+k^2(k^2+k)+(k+1)k^2(k+1)+0+0+0+0+(k+1)(2k+1)k=k(k+1)^2(2k+1).\]
Next, suppose that $2j+1=4i+1$ for some $1\leq i\leq \left\lfloor\frac{k-4}{4}\right\rfloor$. In this case,
\begin{align*}
c(v_{2j+1}) &=(k+1)k(k^2+k+i)+k^2(k^2+k-i)+0+(k+1)i\\
&+0+0+k(2k+1)i+(k+1)(2k+1)(k-i)=k(k+1)^2(2k+1).
\end{align*}
Now, suppose that $2j+1=4i+3$ for some $0\leq i\leq \left\lfloor\frac{k-4}{4}\right\rfloor$. We have 
\begin{align*}
c(v_{2j+1}) &=(k+1)k(k^2+k-i)+k^2(k^2+k+i+1)+0+(k+1)(k-i)\\
&+(k+1)(2k+1)i+k(2k+1)(k-i)+0+0=k(k+1)^2(2k+1)
\end{align*}
To complete the proof of \ref{eq:v1}, it suffices to consider vertices of the form $v_{2j+1}$ where $4\left\lfloor \frac{k-4}{4}\right\rfloor+3< 2j+1\leq k$. 
If $k\equiv0\bmod 4$ and $k\geq4$, then there is no case to check. Now, suppose that $k\equiv2\bmod 4$ and $k\geq6$. We need to consider the case $2j+1=k-1$. We have
\begin{align*}
c(v_{k-1}) &=(k+1)\frac{(4k^2+5k-2)k}{4}+k\frac{k^2(2k+1)}{2}+(k+1)\frac{(k+2)k}{4}+(k+1)\frac{k}{2}+0\\
&+0+k\frac{(2k+1)k}{2}+(k+1)\frac{(2k+1)k}{2}=k(k+1)^2(2k+1).
\end{align*}
This completes the proof of Lemma~\ref{lem:vertexCovered}.
\end{proof}

Finally, we prove Lemma~\ref{lem:edgeCovered}.

\begin{proof}[Proof of Lemma~\ref{lem:edgeCovered}]
Note that, by the symmetry of the construction in Definition~\ref{defn:sequences}, we need only consider edges in which the minimum index of a vertex contained in the edge is between $0$ and $k$. The proof is divided into cases based on the parity of $k$, handling small values of $k$ separately.
\begin{casee}
$k=1$. 
\end{casee}
We have
\begin{align*}c(v_{0},v_{1})= 0+2 \cdot 4+2\cdot 2+ 0+0=12=k(k+1)^2(2k+1),
\end{align*}
verifying \ref{eq:e0} and we have
\begin{align*}c(v_{1},v_{2})= 1\cdot 4+1 \cdot 4+2\cdot 2+0+0+0+0=12=k(k+1)^2(2k+1),
\end{align*}
verifying \ref{eq:e1}.
\begin{casee}
$k=2$. 
\end{casee}
We have
\begin{align*}c(v_{0},v_{1})= 0+ 3 \cdot 16+3 \cdot 14 +0+0=90=k(k+1)^2(2k+1)
\end{align*}
and 
\begin{align*}c(v_{2},v_{3})=3 \cdot 10 +3 \cdot 10 + 0 +3 \cdot 5 + 3 \cdot 5 =90=k(k+1)^2(2k+1), \end{align*}
verifying \ref{eq:e0}. Further, we have
\begin{align*}c(v_{1},v_{2})=2 \cdot 16 + 2 \cdot 10+3 \cdot 1 + 0 + 0+ 2 \cdot 2 \cdot 5 + 3 \cdot 5=90=k(k+1)(2k+1), \end{align*}
verifying \ref{eq:e1} in the case $k=2$.

\begin{casee}
$k=3$.  
\end{casee}
We have
\begin{align*}c(v_{0},v_{1}) = 0+4 \cdot 48+ 4 \cdot 36+0+0 =336=k(k+1)^2(2k+1)\end{align*}
and
\begin{align*}c(v_{2},v_{3}) =4 \cdot 40+4 \cdot 28+4 \cdot 2+0+4 \cdot 14=336=k(k+1)^2(2k+1),
\end{align*}
verifying \ref{eq:e0}. Further, we have
\begin{align*}
c(v_{1},v_{2}) =& 3 \cdot 48+3 \cdot 40+4 \cdot 4+0+0+0+4 \cdot 14=336= k(k+1)^2(2k+1)
\end{align*}
and
\begin{align*}
c(v_{3},v_{4}) =& 3\cdot 28+3 \cdot 28 +0+0+2 \cdot 3 \cdot 14 + 2 \cdot 3 \cdot 14+0=336=
k(k+1)^2(2k+1),
\end{align*}
verifying \ref{eq:e1} in the case $k=3$.
\begin{casee}
$k$ is odd and $k\geq 5$. 
\end{casee}
We start by proving \ref{eq:e0}. First, consider the case $j=0$. In this case,
\[c(v_{0},v_{1}):=0 + (k+1)^2(k(k+1)) + (k+1)k^2(k+1) + 0 + 0=k(k+1)^2(2k+1).\]
Next, suppose that $2j=4i$ for some $1\leq i\leq \left\lfloor\frac{k-5}{4}\right\rfloor$. We have
\begin{align*}
c(v_{2j},v_{2j+1}) =& (k+1)(k+1)(k^2-i) + (k+1)(k+1)(k(k+1)-i) + (k+1)i\\
  +& (k+1)(2k+1)i + 0= k(k+1)^2(2k+1).
\end{align*}
Suppose that $2j=4i+2$ for some $0\leq i\leq \left\lfloor\frac{k-5}{4}\right\rfloor$. Then
\begin{align*}
c(v_{2j},v_{2j+1}) =& (k+1)(k+1)(k^2+i) + (k+1)(k+1)(k^2-k+i) + (k+1)(k-i)\\
   +& 0 + (k+1)(2k+1)(k-i)= k(k+1)^2(2k+1).
\end{align*}
To complete the proof of \ref{eq:e0}, it suffices to consider the edges of the form $v_{2j}v_{2j+1}$ where $4\left\lfloor\frac{k-1}{4}\right\rfloor\leq 2j\leq k$. 
First, suppose that $k\equiv1\bmod 4$ and $k\geq5$. In this case, all that remains is to consider $2j=k-1$. We have
\begin{align*}
c(v_{k-1},v_{k}) &= (k+1)\left(\frac{4k^3+3k^2+1}{4}\right) + (k+1)\frac{4k^3+7k^2+4k+1}{4}
 + (k+1)\frac{k-1}{4}& \\
 &+ (k+1)\left(\frac{(k-1)(2k+1)}{4}\right) + 0=k(k+1)^2(2k+1).
\end{align*}
Now, suppose that $k\equiv3\bmod 4$ and $k\geq7$. In this case, all that remains is to consider $2j=k-3$ and $2j=k-1$. We have
\begin{align*}
c(v_{k-3},v_{k-2}) &=(k+1) \frac{4k^3+3k^2+2k+3}{4}+(k+1)\frac{4k^3+7k^2+6k+3}{4}+(k+1)\frac{k-3}{4}\\
&+(k+1)(2k+1)\frac{k-3}{4}+0=k(k+1)^2(2k+1)
\end{align*}
and
\begin{align*}
c(v_{k-1},v_{k}) &=(k+1)\frac{2k^3+3k^2-1}{2}+(k+1)\frac{2k^3+k^2-2k-1}{2}+(k+1)\frac{k+1}{2}+0\\
&+(k+1)\frac{2k^2+3k+1}{2}=k(k+1)^2(2k+1).
\end{align*}
We now focus on \ref{eq:e1}. Suppose that $2j+1=4i+1$ for some $0\leq i\leq \left\lfloor\frac{k-5}{4}\right\rfloor$. We have
\begin{align*}
c(v_{2j+1},v_{2j+2}) =& k(k+1)(k^2+k-i) + k(k+1)(k^2+i) + 0 + (k+1)(2k+1)i\\
   +& 0 + 0+(k+1)(2k+1)(k-i)= k(k+1)^2(2k+1).
\end{align*}
Suppose now that $2j+1=4i+3$ for some $0\leq i\leq \left\lfloor\frac{k-5}{4}\right\rfloor$. Then
\begin{align*}
c(v_{2j+1},v_{2j+2}) =& k(k+1)(k^2-k+i) + k(k+1)(k^2-i-1) + 0 + 0 + 2k(2k+1)(k-i) \\
  +& 2k(2k+1)(i+1) + 0 = k(k+1)^2(2k+1).
\end{align*}
To complete the proof of \ref{eq:e1}, it suffices to consider the edges of the form $v_{2j+1}v_{2j+2}$ where $4\left\lfloor\frac{k-5}{4}\right\rfloor+3< 2j+1\leq k$.
Suppose that $k\equiv1\bmod 4$ and $k\geq5$. In this case, we need to consider $2j+1=k$. We have 
\begin{align*}
c(v_{k},v_{k+1}) =& k\frac{4k^3+7k^2+4k+1}{4}+k\frac{4k^3+7k^2+4k+1}{4}+(k+1)\frac{(k+1)^2}{2}\\
+&(k+1)\frac{(k-1)(2k+1)}{4}+0+0+(k+1)\frac{(k-1)(2k+1)}{4}=k(k+1)^2(2k+1).
\end{align*}
Suppose that $k\equiv3\bmod 4$ and $k\geq7$. In this case, we need to consider $2j+1\in\{k-2,k\}$. We have  
\begin{align*}
c(v_{k},v_{k+1}) &= k \frac{2k^3+k^2-2k-1}{2}+k\frac{2k^3+k^2-2k-1}{2} + 0+0\\
&+2k\frac{2k^2+3k+1}{2}+2k\frac{2k^2+3k+1}{2}+0 =k(k+1)^2(2k+1).
\end{align*}

\begin{casee}
$k$ is even and $k\geq 4$. 
\end{casee}
We start by proving \ref{eq:e0}. First, consider the case $j=0$. In this case,
\[c(v_{0},v_{1}):=0 + (k+1)^2(k(k+1)) + (k+1)k^2(k+1) + 0 + 0=k(k+1)^2(2k+1).\]
Next, suppose that $2j=4i$ for some $1\leq i\leq \left\lfloor\frac{k-4}{4}\right\rfloor$. We have
\begin{align*}
c(v_{2j},v_{2j+1}) =(k+1)k(k^2+k+i) + (k+1)k((k+1)^2-i)+0+0+0=k(k+1)^2(2k+1).
\end{align*}
Suppose that $2j=4i+2$ for some $0\leq i\leq \left\lfloor\frac{k-4}{4}\right\rfloor$. Then
\begin{align*}
c(v_{2j},v_{2j+1}) &= (k+1)k(k^2+k-i) + (k+1)k(k^2+i) + 0\\
&+(k+1)(2k+1)i+(k+1)(2k+1)(k-i)=k(k+1)^2(2k+1).
\end{align*}
To complete the proof of \ref{eq:e0}, it suffices to consider the edges of the form $v_{2j}v_{2j+1}$ where $4\left\lfloor\frac{k}{4}\right\rfloor\leq 2j\leq k$. 
First, suppose that $k\equiv0\bmod 4$ and $k\geq4$. In this case, all that remains is to consider $2j=k$. We have
\begin{align*}
c(v_{k},v_{k+1}) &=(k+1)k^2\left(k+\frac{5}{4}\right)+(k+1)k^2\left(k+\frac{5}{4}\right)\\
&+(k+1)\frac{(k+2)k}{2}+0+0=k(k+1)^2(2k+1).
\end{align*}
Now, suppose that $k\equiv2\bmod 4$ and $k\geq6$. In this case, all that remains is to consider $2j=k-2$ and $2j=k$. We have
\begin{align*}
c(v_{k-2},v_{k-1}) &=(k+1)\frac{k(4k^2+5k-2)}{4}+(k+1)\frac{k(2k^2+3k+2)}{2}\\
&+(k+1)\frac{(k+2)k}{4}+0+0=k(k+1)^2(2k+1).
\end{align*}
and
\begin{align*}
c(v_{k},v_{k+1}) &=(k+1)\frac{k^2(2k+1)}{2}+(k+1)\frac{k^2(2k+1)}{2}+0\\
&+(k+1)\frac{(2k+1)k}{2}+(k+1)\frac{(2k+1)k}{2}=k(k+1)^2(2k+1).
\end{align*}

We now focus on \ref{eq:e1}. Suppose that $2j+1=4i+1$ for some $0\leq i\leq \left\lfloor\frac{k-4}{4}\right\rfloor$. We have
\begin{align*}
c(v_{2j+1},v_{2j+2}) &= k^2((k+1)^2-i) + k^2(k^2+k-i) + 
 (k+1)i+0+0\\
 &+2k(2k+1)i+(k+1)(2k+1)(k-i)=k(k+1)^2(2k+1).
\end{align*}
Suppose now that $2j+1=4i+3$ for some $0\leq i\leq \left\lfloor\frac{k-4}{4}\right\rfloor$. Then
\begin{align*}
c(v_{2j+1},v_{2j+2}) &= k^2(k^2+i) + k^2(k^2+k+i+1) + (k+1)(k-i)+ (k+1)(2k+1)i\\
&+2k(2k+1)(k-i)+0+0 = k(k+1)^2(2k+1).
\end{align*}
To complete the proof of \ref{eq:e1}, it suffices to consider the edges of the form $v_{2j+1}v_{2j+2}$ where $4\left\lfloor\frac{k-4}{4}\right\rfloor+3< 2j+1\leq k$.
In the case $k\equiv0\bmod 4$ and $k\geq4$, there is nothing to confirm.
Finally, suppose that $k\equiv2\bmod 4$ and $k\geq6$. In this case, we need to consider $2j+1=k-1$. We have  
\begin{align*}
c(v_{k-1},v_{k}) 
&=k\frac{(2k^2+3k+2)k}{2}+k\frac{k^2(2k+1)}{2}+(k+1)\frac{k}{2}+0+0\\
&+2k\frac{(2k+1)k}{2}+(k+1)\frac{(2k+1)k}{2}=k(k+1)^2(2k+1),
\end{align*}
completing the proof of Lemma~\ref{lem:edgeCovered}.
\end{proof}

\end{document}